\documentclass[graybox]{svmult}

\usepackage{helvet}         
\usepackage{courier}        
\usepackage{type1cm}        
\usepackage{makeidx}         
\usepackage{graphicx}        
\usepackage{multicol}        
\usepackage[bottom]{footmisc}
\usepackage{bm}
\usepackage{mathtools}

\makeindex             

\usepackage{latexsym}
\usepackage{dsfont}
\usepackage{bbm}
\usepackage{amssymb}
\usepackage{amsmath}

\newtheorem{Theorem}{Theorem}[section]
\newtheorem{Lemma}[Theorem]{Lemma}

\newtheorem{Prop}[Theorem]{Proposition}
\newtheorem{Rem}[Theorem]{Remark}
\newtheorem{Exa}[Theorem]{Example}

\usepackage[top=3.cm, bottom=3cm, outer=3cm, inner=2.8cm]{geometry}

\def\cB{\mathcal{B}}
\def\cC{\mathcal{C}}

\def\cG{\mathcal{G}}

\def\cL{\mathcal{L}}

\def\cP{\mathcal{P}}

\def\cS{\mathcal{S}}

\def\Erw{\mathbb{E}}

\def\N{\mathbb{N}}
  
\def\Prob{\mathbb{P}} 

\def\R{\mathbb{R}}

\def\Z{\mathbb{Z}}

\def\eps{\varepsilon}

\def\1{\mathbf{1}}
\def\3{{\ss}}

\def\eqdist{\stackrel{d}{=}}
\def\weakeq{\stackrel{d}{\simeq}}

\def\idist{\stackrel{d}{\to}}

\def\weakly{\stackrel{w}{\to}}
\def\RA{\Rightarrow}

\def\LRA{\Leftrightarrow}
\def\wh{\widehat}

\def\ovl{\overline}

\def\sign{\textsl{sign}}

\allowdisplaybreaks[4] 

\begin{document}

\title*{Stability of perpetuities in Markovian environment}
\titlerunning{Stability of perpetuities in Markovian environment}
\author{Gerold Alsmeyer and Fabian Buckmann}
\institute{Inst.~Math.~Stochastics, Department
of Mathematics and Computer Science, University of M\"unster,
Orl\'eans-Ring 10, D-48149 M\"unster, Germany.\at
\email{gerolda@math.uni-muenster.de, f\_buck01@uni-muenster.de}\at
Both authors were partially supported by the Deutsche Forschungsgemeinschaft (SFB 878).}

\maketitle

\abstract{The stability of iterations of affine linear maps $\Psi_{n}(x)=A_{n}x+B_{n}$, $n=1,2,\ldots$, is studied in the presence of a Markovian environment, more precisely, for the situation when $(A_{n},B_{n})_{n\ge 1}$ is modulated by an ergodic Markov chain $(M_{n})_{n\ge 0}$ with countable state space $\cS$ and stationary distribution $\pi$. We provide necessary and sufficient conditions for the a.s. and the distributional convergence of the backward iterations $\Psi_{1}\circ\ldots\circ\Psi_{n}(Z_{0})$ and also describe all possible limit laws as solutions to a certain Markovian stochastic fixed-point equation. As a consequence of the random environment, these limit laws are stochastic kernels from $\cS$ to $\R$ rather than distributions on $\R$, thus reflecting their dependence on where the driving chain is started. We give also necessary and sufficient conditions for the distributional convergence of the forward iterations $\Psi_{n}\circ\ldots\circ\Psi_{1}$. The main differences caused by the Markovian environment as opposed to the extensively studied case of independent and identically distributed (iid) $\Psi_{1},\Psi_{2},\ldots$ are that: (1) backward iterations may still converge in distribution if a.s. convergence fails, (2) the degenerate case when $A_{1}c_{M_{1}}+B_{1}=c_{M_{0}}$ a.s. for suitable constants $c_{i}$, $i\in\cS$, is by far more complex than the degenerate case for iid $(A_{n},B_{n})$ when $A_{1}c+B_{1}=c$ a.s. for some $c\in\R$, and (3) forward and backward iterations generally have different laws given $M_{0}=i$ for $i\in\cS$ so that the former ones need a separate analysis. Our proofs draw on related results for the iid-case, notably by Vervaat \cite{Vervaat:79}, Grincevi\v{c}ius \cite{Grincev:82}, and Goldie and Maller \cite{GolMal:00}, in combination with recent results by the authors \cite{AlsBuck:16} on fluctuation theory for Markov random walks.}

\bigskip

{\noindent \textbf{AMS 2000 subject classifications:}
60J10 (60H25 60J15 60K05 60K15) \ }

{\noindent \textbf{Keywords:} random affine map, Markov modulation, iterated function system, forward and backward iteration, perpetuity, a.s. and distributional convergence, stochastic fixed-point equation}

\section{Introduction}\label{sec:intro}

The purpose of this article is to study stability aspects of iterations of random affine linear maps 
$$ \Psi_{n}(x)\ =\  A_{n}x+B_{n},\quad x\in\R, $$ 
in a discrete Markovian environment, that is for a sequence $(A_{n}, B_{n})_{n\ge 1}$ of $\R^{2}$-valued random vectors which is modulated by an ergodic (positive recurrent and aperiodic) Markov chain $(M_{n})_{n\ge 0}$ with countable state space $\cS$, transition matrix $P=(p_{ij})_{i,j\in\cS}$ and unique stationary law $\pi$. This means that, conditioned upon $M_{0}=i_{0},\,M_{1}=i_{1},\ldots$ for arbitrary $i_{0},i_{1},\ldots\in\cS$,
\begin{itemize}
\item $(A_{1},B_{1}),\,(A_{2},B_{2}),...$ are conditionally independent,
\item the conditional law of $(A_{n},B_{n})$ depends only on $(i_{n-1},i_{n})$ and is temporally homogeneous, i.e. 
$$ \Prob((A_{n},B_{n})\in\cdot|M_{n-1}=i_{n-1},\,M_{n}=i_{n})\ =\ K_{i_{n-1}i_{n}} $$ 
for a stochastic kernel $K$ from $\cS^{2}$ to $\R^{2}$ and all $n\ge 1$.
\end{itemize}
Our goal is to provide necessary and sufficient conditions for the convergence in distribution of the iterated function system (IFS)
\begin{equation}\label{ifsf}
R_{n}\ :=\ \Psi_{n}(R_{n-1})\ =\ \Psi_{n}\circ\ldots\circ\Psi_{1}(R_{0}),\quad n=1,2\ldots,
\end{equation} 
also called forward iterations, as well as conditions for the almost sure convergence of the corresponding backward iterations
\begin{equation}\label{backward sequence}
\Psi_{1}\circ\ldots\circ\Psi_{n}(R_{0})\ =\ \Pi_{n}R_{0}\ +\ \sum_{k=1}^{n}\Pi_{k-1}B_k,\quad n=1,2,\ldots
\end{equation}
where
$$ \Pi_{0}\ :=\ 1\quad\text{and}\quad\Pi_{n}\ :=\ A_{1} A_{2} \cdot\ldots\cdot A_{n},\quad n=1,2,\ldots $$
and $R_{0}$ and $(M_{n},A_{n},B_{n})_{n\ge 1}$ are conditionally independent given $M_{0}$. Under the last assumption, we call $R_{0}$ an \emph{admissible initial value} or just \emph{admissible}, for it ensures that $(M_{n},R_{n})_{n\ge 0}$ forms a temporally homogeneous Markov chain, its transition kernel being
$$ \Prob(M_{1}=j,R_{1}\in\cdot|M_{0}=i,R_{0}=r)\ =\ p_{ij}\,\Prob(A_{1}r+B_{1}\in\cdot|M_{0}=i,M_{1}=j)\quad\text{a.s.} $$
for all $i,j\in\cS$ and $r\in\R$. If $R_{0}=0$, the backward iterations take the form
\begin{equation}\label{sequence}
Z_{n}\ :=\ \Psi_{1}\circ\ldots\circ\Psi_{n}(0)\ =\ \sum_{k=1}^{n}\Pi_{k-1}B_{k},\quad n=1,2,\ldots
\end{equation}
with limiting random variable (if it exists)
\begin{equation}\label{per}
Z_{\infty}\ :=\ \sum_{k\ge 1}\Pi_{k-1}B_{k},
\end{equation} 
often called perpetuity due to its interpretation as a sum of perpetual discounted payments
in the realm of insurance and  finance. For a stationary and ergodic sequence $(A_{n},B_{n})_{n\ge 1}$ with generic copy $(A,B)$, it was shown by Brandt \cite{Brandt:86} that the sum in \eqref{per} does indeed converge absolutely if
\begin{equation}\label{Brandt cond}
\Erw\log|A|<0\quad\text{and}\quad\Erw\log^{+}|B|<\infty.
\end{equation}
He further showed under \eqref{Brandt cond} that the sequence
$$ Z_{\infty}(n)\ :=\ \sum_{k\ge n}\left(\prod_{l=1}^{k-1}A_{n+l-1}\right)B_{n+k-1},\quad n\ge 1, $$
thus $Z_{\infty}=Z_{\infty}(1)$, is the only proper stationary sequence satisfying
$$ Z_{\infty}(n)\ =\ \Psi_{n}(Z_{\infty}(n+1))\ =\ A_{n}Z_{\infty}(n+1)+B_{n}\quad\text{a.s.} $$
for all $n\ge 1$. In the Markov-modulated situation described above, our results will show that \eqref{Brandt cond} is far from being necessary for Brandt's conclusion to be valid.

\subsubsection*{The iid-case}

The case when $(A_{1},B_{1}),(A_{2},B_{2}),...$ are independent and identically distributed (iid) has received by far the most attention in the past, and a good account of the substantial literature may be found in the recent monography by Buraczewski et al. \cite{BurDamMik:16}. Here we only mention the work by Vervaat \cite{Vervaat:79}, Goldie \cite{Goldie:91}, Grincevi\v{c}ius \cite{Grincev:81,Grincev:82}, Goldie and Gr\"ubel \cite{GolGru:96}, Goldie and Maller \cite{GolMal:00}, Alsmeyer et al. \cite{AlsIksRoe:09} and, last but not least, the celebrated work by Kesten \cite{Kesten:73} on the multivariate case (not treated here) when the $A_{n}$ are $d\times d$ matrices and the $B_{n}$ are random vectors in $\R^{d}$.

\vspace{.1cm}
Note that $Z_{0}$ is admissible in the iid-case iff it is independent of $(A_{n},B_{n})_{n\ge 1}$. Due to the simple observation that
\begin{equation}\label{forward eqdist backward}
\Psi_{1}\circ\ldots\circ\Psi_{n}(Z_{0})\ \eqdist\ \Psi_{n}\circ\ldots\circ\Psi_{1}(Z_{0})
\end{equation}
for any $n\ge 1$ and admissible $Z_{0}$, where $\eqdist$ means equality in law, distributional convergence of the forward and backward iterations are equivalent, and one may therefore focus on the backward iterations. Moreover, the convergence of this sequence then even holds in the almost sure sense, the limit being the perpetuity $Z_{\infty}$ defined in \eqref{per}. Necessary and sufficient conditions for this convergence, i.e., for the existence of $Z_{\infty}$ as a proper random variable have been provided by Goldie and Maller \cite[Thm.~2.1]{GolMal:00}. Below we state a slightly stronger version of their theorem which allows for a better comparison with our corresponding result in the Markov-modulated case (Theorem \ref{thm:a.s. convergence}).

\vspace{.1cm}
For $x>0$, define
\begin{align*}
J(0)\,:=\,1\quad\text{and}\quad J(x)\,:=\,
\begin{cases}
\displaystyle{\frac{x}{\Erw(X^{+}\wedge x)}},&\text{if }\Prob(X>0)>0,\\[2mm]
\hfill x,&\text{otherwise}.
\end{cases}
\end{align*}
We further put $\Psi_{k:n}:=\Psi_{k}\circ\ldots\circ\Psi_{n}$ and $\Psi_{n:k}:=\Psi_{n}\circ\ldots\circ\Psi_{k}$ for $1\le k\le n$, and stipulate $\log^{+}0=0$.

\begin{Theorem}\label{thm:G/M theorem}
Suppose that
\begin{equation}\label{nonzero}
\Prob\{A=0\}\,=\,0\quad\text{and}\quad\Prob\{B=0\}\,<\,1.
\end{equation}
Then the following assertions are equivalent:
\begin{description}[(d)]\itemsep2pt
\item[(a)] $\Psi_{1:n}(Z_{0})$ converges a.s. to a proper random variable for any initial variable $Z_{0}$.
\item[(b)] $\lim_{n\to\infty}\Psi_{1:n}(0)=Z_{\infty}=\sum_{n\ge 1} \Pi_{n-1}B_{n}$ a.s. and $Z_\infty$ is a proper random variable. 
\item[(c)] $\lim_{n\to\infty} \Pi_n=0$ a.s. and $\Erw\, J(\log^{+} |B|)<\infty$.
\item[(d)] $\Prob(|A|=1)<1$ and $\limsup_{n\to\infty}|\Pi_{n-1}\,B_n|<\infty$ a.s.
\item[(e)] $\lim_{n\to\infty} \Pi_{n-1}\,B_n=0$ a.s.
\item[(f)] $\sum_{n\geq 1} |\Pi_{n-1} \,B_n|<\infty$ a.s.
\end{description}
On the other hand, if these conditions fail, then
\begin{equation}\label{DegBed ordinary}
\Prob(B=c\,(1-A))\ <\ 1\quad\text{ for all }c\in\R,
\end{equation} 
and 
$$ |\Psi_{1:n}(Z_0)|\ \xrightarrow{\Prob}\ \infty\quad\text{for any admissible }Z_{0} $$
are equivalent.
\end{Theorem}

Goldie and Maller actually showed ``(c)$\LRA$(d)$\LRA$(e)$\LRA$(f)$\RA$(b)'' and the ``if''-part of the last assertion of the theorem. But its ``only if''-part (note that $\Psi_{1:n}(c)=c$ if $B=c(1-A)$ a.s.) as well as ``(a)$\RA$(b)'' and ``(b)$\RA$(d)'' are obvious.

\vspace{.1cm}
So we see that, under the nondegeneracy condition \eqref{DegBed ordinary}, the backward iteration $Z_{n}$ either converges a.s. or diverges to $\infty$ in probability. The conditions on $A$ and $B$ in \eqref{nonzero}, which are always assumed to hold in the subsequent discussion, rule out trivial cases. This being clear for the condition on $B$, we only note that $\Prob(A=0)>0$ implies that $N=\inf\{n\ge 1:A_{n}=0\}$ is a.s. finite and thus $Z_{\infty}=\sum_{k=1}^{N}\Pi_{k-1}B_{k}$ a.s.

\vspace{.1cm}
Provided that $Z_{\infty}$ is a proper random variable, we have $\Pi_{n}\to 0$ a.s. by Theorem \ref{thm:G/M theorem}(c). Consequently,
$$ \lim_{n\to\infty}|\Psi_{1:n}(Z_{0})-\Psi_{1:n}(0)|\ =\ \lim_{n\to\infty}|\Pi_{n}Z_{0}|\ =\ 0\quad\text{a.s.} $$
for any initial value $Z_{0}$ which in combination with \eqref{forward eqdist backward} entails that the law of $Z_{\infty}$ equals the \emph{unique} stationary distribution of the forward iterations
$R_{n}$, clearly a recursive Markov chain (see \eqref{ifsf}), and thus a distributional fixed point of the equation
\begin{equation}\label{SFPE iid}
R\ \eqdist\ AR'+B
\end{equation}
under the usual convention that the variable $R'$ is a copy of $R$ and independent of $(A, B)$, see \cite[Lemma 1.1]{Vervaat:79}. Regarding all solutions to this equation, we quote the following result by Vervaat \cite[Thm.~4.5]{Vervaat:79} and Goldie and Maller \cite[Thm.~3.1]{GolMal:00}. Let $\cP(\R)$ denote the set of probability distributions on $\R$.

\begin{Theorem}\label{thm:fixed points IID}
Suppose $\Prob(A=0)=0$. Then there exists a fixed point $Q\in\cP(\R)$ of \eqref{SFPE iid} iff one of the following conditions is satisfied:
\begin{description}[(b)]\itemsep2pt
\item[(a)] $\lim_{n\to\infty}\Pi_{n}=0$ a.s. and $\Erw\,J(\log^{+}|B|)<\infty$. In this case, $Q$ is unique and equals the law of $Z_{\infty}$.
\item[(b)] $\Prob(|A|=1)=\Prob(B=c\,(1-A))=1$ for some $c\in\R$. Then,
\begin{description}\itemsep2pt
	\item[(b.1)] if $\Prob(A=1)<1$, any $Q$ which is symmetric about $c$ is a fixed point,
	\item[(b.2)] if $\Prob(A=1)=1$, any $Q\in\cP(\R)$ is a fixed point. 
\end{description}
\item[(c)] $\limsup_{n\to\infty}\,|\Pi_n|=\infty$ a.s. and $\Prob(B=c\,(1-A))=1$. In this case, $Q=\delta_{c}$ is the unique fixed point.
\end{description}
\end{Theorem}

Returning to the Markov-modulated situation when $(A_{n},B_{n})_{n\ge 1}$ is governed by an ergodic discrete Markov chain $(M_{n})_{n\ge 0}$ with unique stationary law $\pi$, it is natural to ask for extensions of the previous two theorems. This appears to be an open question despite a number of contributions by de Saporta \cite{Saporta:05b}, Roitershtein \cite{Roitershtein:07}, Collamore \cite{Collamore:09}, Ghosh et al \cite{Ghosh+al:10}, Hay et al \cite{HayRasRoi:11}, Buraczewski and Letachowicz \cite{BurLet:12}, Basu and Roitershtein \cite{BasuRoi:13} dealing with other aspects of the model, mostly the tail of $Z_{\infty}$ under varying assumptions (including continuous state space) on the driving chain $(M_{n})_{n\ge 0}$. Applications in Econometrics can be found in Hamilton \cite{Hamilton:89}, Benhabib et al \cite{Benhabib+al:11}, Benhabib and Dave \cite{BenhabibDave:13}.

\vspace{.1cm}
We will use the common notation $\Prob_{i}:=\Prob(\cdot|M_{0}=i)$ for $i\in\cS$ and $\Prob_{\lambda}=\sum_{i\in\cS}\lambda_{i}\,\Prob_{i}$ for any distribution $\lambda=(\lambda_{i})_{i\in\cS}$ on $\cS$. Since $(A_{n},B_{n})_{n\ge 1}$ then forms a stationary sequence under $\Prob=\Prob_{\pi}$, Brandt's result applies to give that $Z_{\infty}$ defined by \eqref{per} exists in the almost sure sense if \eqref{Brandt cond} holds with $(A,B)$ denoting a generic copy of $(A_{1},B_{1})$ under $\Prob_{\pi}$. Regarding the recursive (and now Markov-modulated) IFS $(R_{n})_{n\ge 0}$, the very same condition further ensures that this IFS has negative (top) Liapunov exponent under $\Prob_{\pi}$ and therefore, by Elton's theorem \cite[Theorem 3]{Elton:90}, a unique stationary law, viz. the $\Prob_{\pi}$-distribution of
\begin{equation}\label{Zinftytag}
{}^{\#}Z_{\infty}\ =\ \sum_{n\ge 0}\left(\prod_{k=0}^{n-1}A_{-k}\right)B_{-n}
\end{equation}
when $(A_{n},B_{n})_{n\in\Z}$ denotes a doubly infinite stationary extension. On the other hand, it should be clear in view of Theorem \ref{thm:G/M theorem} that necessary and sufficient conditions for the existence of $Z_{\infty}$ in the almost sure sense are more difficult to come by in the Markov-modulated situation. Our results to be stated in Section \ref{sec:main results} will actually also show that there are nondegenerate situations where $Z_{n}$ does not converge a.s. but still converges in distribution, which is impossible for iid $(A_{n},B_{n})$ (see Theorem \ref{thm:weak convergence}).

\vspace{.1cm}
Regarding forward versus backward iterations, the following is another important aspect that distinguishes the Markov-modulated case from the iid-case. If $(M_{n},A_{n+1},B_{n+1})_{n\in\Z}$ denotes a doubly infinite extension of the stationary sequence $(M_{n},A_{n+1},B_{n+1})_{n\ge 0}$ under $\Prob_{\pi}$ or, equivalently, $(M_{n},\Psi_{n+1})_{n\in\Z}$ the resulting doubly infinite extension of $(M_{n},\Psi_{n+1})_{n\ge 0}$, then it is no longer always true that
$$ (\Psi_{1},\ldots,\Psi_{n})\ \eqdist\ (\Psi_{n},\ldots,\Psi_{1})\quad\text{under }\Prob_{\pi} $$
for all $n\in\N$. Indeed, it requires $(M_{n})_{n\ge 0}$ to be reversible. In general, however, the dynamics of the backward sequence $(M_{-n},\Psi_{-n+1})_{n\ge 0}$ are different due to the fact that the backward driving chain $(M_{-n})_{n\ge 0}$ has the dual transition matrix ${}^{\#}P=(\pi_{j}p_{ji}/\pi_{i})_{i,j\in\cS}$. More precisely, if $({}^{\#}M_{n},{}^{\#}A_{n+1},{}^{\#}B_{n+1})_{n\ge 0}$ denotes a dual of $(M_{n},A_{n+1},B_{n+1})_{n\ge 0}$ with $M_{0}={}^{\#}M_{0}$, ${}^{\#}\Psi_{n}(x):={}^{\#}A_{n}x+{}^{\#}B_{n}$ for $n\ge 1$ and such that both sequences are stationary under $\Prob_{\pi}$, see Subsection \ref{subsec:duality} for further details, then we have, under $\Prob_{\pi}$,
\begin{align}\label{eq:forward<->backward dual}
\Psi_{1:n}(Z_{0})\ \eqdist\ {}^{\#}\Psi_{n:1}(Z_{0})\quad\text{and}\quad\Psi_{n:1}(Z_{0})\ \eqdist\ {}^{\#}\Psi_{1:n}(Z_{0})
\end{align}
whenever $Z_{0}$ is admissible for $(M_{n},A_{n+1},B_{n+1},{}^{\#}M_{n},{}^{\#}A_{n+1},{}^{\#}B_{n+1})_{n\ge 0}$, i.e. conditionally independent of this sequence given $M_{0}={}^{\#}M_{0}$. This suggests that limit results for the forward iterations $\Psi_{n:1}(Z_{0})$ may still be derived by a look at backward iterations, but for the dual sequence $({}^{\#}\Psi_{n})_{n\ge 1}$. On the other hand, a nonconstant $Z_{0}$ may no longer be admissible for $({}^{\#}M_{n},{}^{\#}A_{n+1},{}^{\#}B_{n+1})_{n\ge 0}$ and \eqref{eq:forward<->backward dual} does no longer hold under $\Prob_{i}$ for $i\in\cS$. Therefore additional arguments will be needed as well.

\vspace{.1cm}
Last but not least, it is to be announced here that Equation \eqref{SFPE iid} as a characterization of the limit law of the backward sequence $\Psi_{1:n}(Z_{0})$ also requires an adjustment in the Markov-modulated case. Without giving details, which are provided in Subsection \ref{subsec:FP property}, we only mention at this point that, when $\Psi_{1:n}(Z_{0})\to Z_{\infty}$ a.s. and thus $\Psi_{2:n}(Z_{0})\to Z_{\infty}'$ a.s. as well, we still have $Z_{\infty}\eqdist Z_{\infty}'$ and also, by the continuity of $\Psi_{1}$, the same formal relation between the two limits as in the iid-case, namely
\begin{equation}\label{SFPE false}
Z_{\infty}\ =\ \Psi_{1}(Z_{\infty}')\ =\ A_{1}Z_{\infty}'+B_{1}\quad\text{a.s.}
\end{equation}
However, \emph{$Z_{\infty}'$ is no longer independent of $(A_{1},B_{1})$}. Roughly speaking, the proper adjustment when aiming to still interpret \eqref{SFPE false} as a stochastic fixed-point equation must be in terms of stochastic kernels, the conditional law of $Z_{\infty}$ given $M_{0}$ then being a solution in a certain sense. With this at hand, we will be able to prove a counterpart of Theorem \ref{thm:fixed points IID}, see Theorem \ref{thm:fixed-point property}.

\vspace{.1cm}
We have organized this work as follows. Some preliminary facts on return times, fluctuation theory for MRW, duality and degeneracy are collected in Section \ref{sec:preliminaries}. The main results are presented in Section \ref{sec:main results}, namely Theorems \ref{thm:a.s. convergence} and \ref{thm:weak convergence} on the convergence of the backward iterations (Subsection \ref{subsec:3.1}), Theorem \ref{thm:fixed-point property} on the solutions to the corresponding stochastic fixed-point equation (Subsection \ref{subsec:FP property}), and Theorem \ref{thm:convergence forward} on the convergence of the forward iterations (Subsection \ref{subsec:convergence forward}). A discussion of the degeneracy condition \eqref{pi-degenerate}, which requires considerably more attention than its counterpart in the iid-case, will be given in Section \ref{sec:degeneracy}, followed by the proofs of the main results in Sections \ref{sec:proof 1st thm}--\ref{sec:proof thm forward}. Finally, two auxiliary lemmata are given in the Appendix.

\section{Preliminaries}\label{sec:preliminaries}

This section is devoted to the collection of some useful  and fundamental definitions, facts and observations related to the sequence $(A_{n},B_{n})_{n\ge 1}$ and its driving chain $(M_{n})_{n\ge 0}$. They will be useful or even needed for the statement of our main results and their proofs.
Let us stipulate for the rest of this article that ''a.s.'' without qualifier means ``$\,\Prob_{\pi}\,$-a.s.'' or, equivalently, ``$\,\Prob_{i}$-a.s. for all $i\in\cS\,$''.

\subsection{Return times}\label{subsec:return times}

Since the driving chain $(M_{n})_{n\ge 0}$ is positive recurrent with discrete state space $\cS$, it is natural to introduce the successive return times to a state $i\in\cS$, viz. $\tau(i):=\tau_{1}(i)$ and
\begin{align*}
\tau_{n}(i)\ :=\ \inf\{k>\tau_{n-1}(i):M_{k}=i\}\quad\text{for }n=1,2,\ldots,
\end{align*}
where $\tau_{0}(i):=0$. We further define
\begin{align*}
&\hspace{2cm}\Psi_{n}^{i}(x)\ :=\ \Psi_{\tau_{n-1}(i)+1:\tau_{n}(i)}(x)\ =\ A_{n}^{i}x+B_{n}^{i}
\shortintertext{and}
&A_{n}^{i}\ :=\ \prod_{k=\tau_{n-1}(i)+1}^{\tau_{n}(i)}A_{k}\quad\text{and}\quad B_{n}^{i}\ :=\ \prod_{k=\tau_{n-1}(i)+1}^{\tau_{n}(i)}\left(\prod_{l=\tau_{n-1}(i)+1}^{k}A_{l}\right)B_{k}.
\end{align*}
for $n\ge 1$ and note that the $(A_{n}^{i},B_{n}^{i})$ are obviously a.s. finite (since all $\tau_{n}(i)$ have this property) and independent (also of $Z_{0}$ if admissible) with
\begin{equation}\label{def A^s,B^s}
(A_{1}^{i},B_{1}^{i})\ =\ \left(\Pi_{\tau(i)},\sum_{k=1}^{\tau(i)}\Pi_{k-1}B_{k}\right).
\end{equation}
They are also identically distributed for $n\ge 2$ under any initial distribution for the driving chain, and even for $n\ge 1$ when choosing $\Prob=\Prob_{i}:=\Prob(\cdot|M_{0}=i)$. Now we have
\begin{align*}
Z_{\tau_{n}(i)}\ =\ \Psi_{1:\tau_{n}(i)}(Z_{0})\ =\ \Psi_{1:n}^{i}(Z_{0})\ =\ \Pi_{\tau(i)}\,\Psi_{2:n}^{i}(Z_{0})\,+\,B_{1}^{i}
\end{align*}
for all $n\ge 1$ and $i\in\cS$ and thus see that convergence of $Z_{\tau_{n}(i)}$ leads back to the iid-case studied by Vervaat \cite{Vervaat:79} and Goldie and Maller \cite{GolMal:00} when using their results for the backsward system $(\Psi_{1:n}^{i}(Z_{0}))_{n\ge 1}$.

\subsection{Fluctuation theory}\label{subsec:fluctuation}

Defining $X_{n}:=-\log|A_{n}|$ and $S_{n}:=-\log|\Pi_{n}|=\sum_{k=1}^{n}X_{k}$ for $n\ge 1$, our assumptions imply that $(M_{n},S_{n})_{n\ge 0}$, with $S_{0}:=0$, forms a zero-delayed Markov random walk (MRW), i.e., $X_{1},X_{2},\ldots$ are conditionally independent given $M_{0},\,M_{1},\ldots$, and the conditional law of $X_{n}$ depends only on $M_{n-1},M_{n}$ and is temporally homogeneous way, thus
$$ \Prob(X_{n}\in\cdot|M_{n-1}=i,\,M_{n}=j)\ =\ F_{ij} $$
for all $n\ge 1$, $i,j\in\cS$ and a stochastic kernel $F$ from $\cS^{2}$ to $\R$.
The following trichotomy is fundamental for our further investigations and a direct consequence of the results in \cite[Section 4]{AlsBuck:16}.

\begin{Prop}\label{prop:trichotomy MRW}
For any MRW $(M_{n},S_{n})_{n\ge 0}$, exactly one of the following three alternatives holds:
\begin{description}[(T3)]\itemsep3pt
\item[(T1)] $S_{n}\to\infty$ a.s.
\item[(T2)] $\Prob_{i}(S_{\tau(i)}=0)=1$ for all $i\in\cS$.
\item[(T3)] $\liminf_{n\to\infty}S_{n}=-\infty$ a.s. and $\Prob_{i}(S_{\tau(i)}=0)<1$ for all $i\in\cS$.
\end{description}
\end{Prop}

If (T1) holds, $(M_{n},S_{n})_{n\ge 0}$ is called \emph{positive divergent}, a particular case being $\Erw_{\pi}X_{1}>0$, which in our setting with $X_{n}=-\log|A_{n}|$ means
$$ \Erw_{\pi}\log|A_{1}|\ <\ 0 $$ 
(as in condition \eqref{Brandt cond} with $\Prob=\Prob_{\pi}$). Type (T2) occurs iff $(M_{n},S_{n})_{n\ge 0}$ is \emph{null-homo\-logous} in the sense of Lalley \cite{Lalley:86}, which means that, for some function $g:\cS\to\R$,
\begin{equation}\label{NH1}
X_{n}\ =\ g(M_{n})-g(M_{n-1})\quad\text{a.s.}
\end{equation}
or, equivalently,
\begin{equation}\label{NH2}
S_{n}\ =\ g(M_{n})-g(M_{0})\quad\text{a.s.}
\end{equation}
for all $n\ge 1$, see \cite[Lemma 4.1]{AlsBuck:16}. This corresponds to the trivial case $S_{n}=0$ a.s. for all $n\ge 0$ in the case of iid increments. Finally, type (T3) occurs when $(M_{n},S_{n})_{n\ge 0}$ is not null-homologous and either \emph{negative divergent}, i.e. $\lim_{n\to\infty}S_{n}=-\infty$ a.s., or \emph{oscillating}, i.e.
$$ -\infty\ =\ \liminf_{n\to\infty}S_{n}\ <\ \limsup_{n\to\infty}S_{n}\ =\ \infty\quad\text{a.s.} $$

Since $(S_{\tau_{n}(i)})_{n\ge 0}$ has iid increments for each $i\in\cS$, the following trichotomy for the embedded sequences $(\Pi_{\tau_{n}(i)})_{n\ge 0}$ follows directly from classical fluctuation theory for ordinary random walks. 

\begin{Prop}\label{prop:trichotomy embedded RW}
Exactly one of the following three alternatives holds for $(\Pi_{\tau_{n}(i)})_{n\ge 0}$:
\begin{description}[(T3')]\itemsep3pt
\item[(T1')] $\Pi_{\tau_{n}(i)}\to 0$ a.s.
\item[(T2')] $\Prob_{i}(|\Pi_{\tau(i)}|=1)=1$.
\item[(T3')] $\limsup_{n\to\infty}|\Pi_{\tau_{n}(i)}|=\infty$ a.s.
\end{description}
Moreover, the type is the same for all $i\in\cS$.
\end{Prop}

We refer to \cite[Lemmata 4.1 and 6.1]{AlsBuck:16} for a proof of the solidarity assertion and note that it implies the equivalence of (T2) and (T2'). On the other hand, there is neither equivalence of (T1) and (T1'), nor of (T3) and (T3'). Namely, $(S_{n})_{n\ge 0}$ may be oscillating although its embedded RW $(S_{\tau_{n}(i)})_{n\ge 0}$ are all positive or negative divergent, see \cite[Example 6.2]{AlsBuck:16}. In other words, (T3) may occur together with (T1') as well as with (T3'). Of course, (T1) always implies (T1').

\subsection{Duality}\label{subsec:duality}

Since $(M_{n},A_{n+1},B_{n+1})_{n\ge 0}$ is stationary under $\Prob_{\pi}$, it can be extended to a doubly infinite stationary sequence $(M_{n},A_{n+1},B_{n+1})_{n\in\Z}$. Given $(M_{n})_{n\in\Z}$, the $(A_{n},B_{n})$ are of course still conditionally independent with the same conditional law as before.
The reversed chain $({}^{\#}M_{n})_{n\ge 0}:=(M_{-n})_{n\ge 0}$, also called dual of $(M_{n})_{n\ge 0}$, has transition matrix
$$ {}^{\#}P\ =\ \left({}^{\#}p_{ij}\right)_{i,j\in\cS}\quad\text{with}\quad{}^{\#}p_{ij}\ =\ \frac{\pi_{j}p_{ji}}{\pi_{i}}. $$
Similarly, the dual of $(M_{n},A_{n+1},B_{n+1})_{n\ge 0}$ (also a Markov chain) is given by
$$ ({}^{\#}M_{n},{}^{\#}A_{n+1},{}^{\#}B_{n+1})_{n\ge 0}=(M_{-n},A_{-n},B_{-n})_{n\ge 0}. $$
The index shift for the $(A,B)$-sequence ensures that its dual counterpart is also Markov-modulated. More precisely, its elements are conditionally independent given the dual chain $({}^{\#}M_{n})_{n\ge 0}$ with 
$$ {}^{\#}K_{ij}\ :=\ \Prob\left(({}^{\#}A_{n},{}^{\#}B_{n})\in\cdot|{}^{\#}M_{n-1}=i,\,{}^{\#}M_{n}=j\right) $$
satisfying the dual relation
$$ {}^{\#}K_{ij}\ =\ K_{ji}. $$
After these settings it is clear that the dual of the IFS generated by $(\Psi_{n})_{n\ge 1}$ is the IFS generated by the dual maps ${}^{\#}\Psi_{n}(t)={}^{\#}A_{n}t+{}^{\#}B_{n}=A_{-n+1}t+B_{-n+1},$ $n\ge 1$. It should then be observed that, when assuming \eqref{Brandt cond} with $\Prob=\Prob_{\pi}$, the random variable ${}^{\#}Z_{\infty}$ defined by \eqref{Zinftytag} may now be rewritten as
$$ {}^{\#}Z_{\infty}\ =\ \sum_{n\ge 1}\left(\prod_{k=1}^{n-1}{}^{\#}A_{k}\right){}^{\#}B_{n} $$
and thus be identified as the a.s. limit of the backward iterations ${}^{\#}\Psi_{1:n}(0)$. As a consequence, the limit law of the forward IFS $(\Psi_{n:1}(0))_{n\ge 1}$ does not generally coincide with limit law of the associated backward IFS $(\Psi_{1:n}(0))_{n\ge 1}$ as in the iid-case. By \eqref{eq:forward<->backward dual}, it rather equals the limit law of its \emph{backward dual} $({}^{\#}\Psi_{1:n}(0))_{n\ge 1}$. Let us finally point out that duality arguments completely fail to apply when studying distributional convergence of the forward sequence under $\Prob_{i}$ for $i\in\cS$ rather that $\Prob_{\pi}$. This is because \eqref{eq:forward<->backward dual} does no longer hold under $\Prob_{i}$ as already mentioned.

\subsection{Degeneracy}\label{subsec:degeneracy}

As mentioned after Theorem \ref{thm:G/M theorem}, degeneracy for iterations of iid random affine maps of generic form $\Psi(t)=At+B$ occurs if $Ac+B=c$ a.s. for some $c\in\R$ (cf. \cite{Vervaat:79} and \cite{GolMal:00}). In the Markov-modulated situation, the corresponding condition looks similar and yet different, namely
\begin{equation}\label{pi-degenerate}
\Prob_{\pi}(A_{1}c_{M_{1}}+B_{1}=c_{M_{0}})\,=\,1\quad\text{for suitable }c_{i}\in\R,\,i\in\cS.
\end{equation}
The condition also appears in \cite[Eq.\~(1.6)]{Roitershtein:07}. Its implications will be discussed in some detail in Section \ref{sec:degeneracy} as they will be of some relevance in connection with our main results. In particular, we will show there (see Props. \ref{prop:degeneracy consequences} and \ref{prop2:degeneracy consequences}) that \eqref{pi-degenerate} is equivalent to
\begin{equation}\label{degenerate}
\Prob_{i}(A_{1}^{i}c_{i}+B_{1}^{i}=c_{i})\ =\ 1\quad\text{for all }i\in\cS\text{ and suitable }c_{i}\in\R,
\end{equation}
with $(A_{1}^{i},B_{1}^{i})$ as in \eqref{def A^s,B^s}. In fact, ``\eqref{pi-degenerate}$\RA$\eqref{degenerate}'' can easily be checked, but the converse requires some work. We will further show (see Lemma \ref{lem:solidarity}) that \eqref{degenerate} already follows whenever $\Prob_{i}(A_{1}^{i}c_{i}+B_{1}^{i}=c_{i})=1$ for just one pair $(i,c_{i})\in\cS\times\R$.

\section{Main results}\label{sec:main results}

For the results below, we make the standing assumption (besides those already stated at the beginning of the Introduction) that (compare \eqref{nonzero})
\begin{equation}\label{standing assumption}
\Prob_{\pi}(A=0)\ =\ 0\quad\text{and}\quad\Prob_{\pi}(B=0)<1,
\end{equation}
where $(A,B)$ denotes a generic copy of the $(A_{n},B_{ n})$ under $\Prob_{\pi}$ which is independent of all other occurring random variables. A brief discussion at the end of this section will show that the situation when \eqref{standing assumption} fails is rather trivial.

\subsection{Convergence results for the backward iterations}\label{subsec:3.1}

Our first two theorems provide necessary and sufficient conditions for the almost sure and distributional convergence (under $\Prob_{i}$ for all $i\in\cS$) of the backward iteration $\Psi_{1:n}(Z_{0})$ for any initial value $Z_{0}$ (admissible in the second result). For $i\in\cS$ and $x\ge 0$, we put
\begin{align*}
&J_{i}(0)\,:=\,1\quad\text{and}\quad J_{i}(x)\,:=\,
\begin{cases}
\displaystyle{\frac{x}{\Erw_{i}(S_{\tau(i)}^{+}\wedge x)}},&\text{if }\Prob_{i}(S_{\tau(i)}\le 0)<1,\\[2mm]
\hfill x,&\text{otherwise},
\end{cases}\\
\shortintertext{and}
&\hspace{2.5cm}W^{i}\,:=\, \max_{1\le k\le\tau(i)}|\Pi_{k-1}B_{k}|.
\end{align*}

\begin{Theorem}\label{thm:a.s. convergence}
Suppose \eqref{standing assumption}. Then the following conditions are equivalent:
\begin{description}[(b)]\itemsep2pt
\item[(a)] $\lim_{n\to\infty} \Psi_{1:n}(0)=Z_{\infty}=\sum_{n\ge 1}\Pi_{n-1}\,B_{n}$ a.s. and $Z_{\infty}$ is a proper random variable. 
\item[(b)] $\lim_{n\to\infty} \Pi_{\tau_{n}(i)}=0$ a.s. and $\Erw_{i}J_{i}(\log^{+}W^{i})<\infty$ for some/all $i\in\cS$.
\item[(c)] $\Prob_{i}(|\Pi_{\tau(i)}|=1)<1$ for some/all $i\in\cS$ and $\limsup_{n\to\infty}|\Pi_{n-1}B_{n}|<\infty$ a.s.
\item[(d)] $\lim_{n\to\infty}\Pi_{n-1}\,B_n=0$ a.s.
\item[(e)] $\sum_{n\ge 1}|\Pi_{n-1} \,B_{n}|<\infty$ a.s.
\end{description}
Furthermore, $\Psi_{1:n}(Z_{0})$ converges a.s. to a proper random variable for any admissible $Z_{0}$ iff $\lim_{n\to\infty}\Pi_{n}=0$ a.s. and 
$\Erw_{i} J_{i}(\log^{+}W^{i})<\infty$ for some/all $i\in\cS$. In this case, the limit always equals $Z_\infty$.
\end{Theorem}

\begin{Rem}\label{rem1:a.s. convergence}\rm
The last assertion of the theorem is easily verified as follows: Since (b) ensures the a.s. convergence of $\Psi_{1:n}(0)$ to $Z_{\infty}=\sum_{n\ge 1}\Pi_{n-1}B_{n}$ and since
$$ \Psi_{1:n}(Z_{0})\ =\ \Pi_{n}\,Z_{0}\ +\ \Psi_{1:n}(0), $$
we see that $\Pi_{n}\to 0$ a.s. does indeed constitute the required extra condition for the asserted equivalence. On the other hand, this condition is not a consequence of (d) as one may expect at first glance. Here is a simple counterexample:
Let $(M_{n})_{n\ge 0}$ be a Markov chain on $\N_{0}$ which, when in state 0, either stays there with probability $p_{00}>0$, or picks a state $i\in\N$ with probability $p_{0i}>0$. When in state $i\in\N$, it always moves back to 0, thus $p_{i0}=1$. In essence, this is the infinite-petal flower chain introduced in \cite[Example 6.2]{AlsBuck:16}, the name being chosen there because the transition graph looks like a flower with infinitely many petals, and it is clearly ergodic. Define further
\begin{align*}
(A_{n},B_{n})\ :=\ 
\begin{cases}
\hfill (1,1),&\text{if }(M_{n-1},M_{n})=(0,0),\\[2pt]
\hfill\big(\exp(p_{0i}^{-1}),1\big),&\text{if }(M_{n-1},M_{n})=(0,i),\,i\in\N\\[2pt]
\big(\exp(-p_{0i}^{-1})/2,\, \exp(-p_{0i}^{-1})\big),&\text{if }(M_{n-1},M_{n})=(i,0),\,i\in\N.
\end{cases}
\end{align*}
Then $\lim_{n\to\infty} \Pi_{\tau_{n}(0)}=0$ a.s. and $\lim_{n\to\infty} \Pi_{n-1}B_{n}=0$ a.s. are obvious, and one can readily verify by a Borel-Cantelli-type argument (cf.~\cite{AlsBuck:16}) that $\limsup_{n\to\infty}\Pi_{n}=\infty$ a.s.

\vspace{.05cm}
Let us finally note that (e) trivially implies (a) which in turn trivially implies (d). Hence, the proof of the theorem is complete if we show the equivalence of assertions (b)--(e).
\end{Rem}

\begin{Rem}\label{rem2:a.s. convergence}\rm
A particular outcome of Theorem \ref{thm:a.s. convergence} is that validity of $\Erw_{i}J_{i}(\log^{+}W^{i})<\infty$ for some $i\in\cS$ implies the very same for all $i\in\cS$. Indeed, the proof of the theorem (see ``(c)$\RA$(b)'') in Section \ref{sec:proof 1st thm} will show that failure of this condition for some $i\in\cS$ always entails $\limsup_{n\to\infty}|\Pi_{n-1}B_{n}|=\infty$ a.s. and thus failure of (d).
\end{Rem}

\vspace{.1cm}
Turning to distributional convergence, the picture changes through the fact that, if a.s. convergence fails, the limit law of $Z_{n}$ may or may not depend on the law of the initial value $Z_{0}$. Before stating the result, let us define
$$ \wh{\tau}(i)\ :=\ \inf\{\tau_{n}(i):\Pi_{\tau_{n}(i)}=1\} $$
for $i\in\cS$ which we will need in the case when $\Prob_{i}(|A_{1}^{i}|=1)=\Prob_{i}(|\Pi_{\tau(i)}|=1)=1$. Due to the aperiodicity of $\tau(i)$ it is not difficult to verify (see Lemma \ref{lem:lattice-type whtau(i)}) that $\wh{\tau}(i)$ is then either of the same type or 2-periodic.

\begin{Theorem}\label{thm:weak convergence}
Suppose \eqref{standing assumption}. Given $i\in\cS$ and an admissible $Z_{0}$, $\Prob_{i}(\Psi_{1:n}(Z_{0})\in\cdot)$ converges weakly to some $Q_{i}$ iff one of the following conditions is fulfilled:
\begin{description}[(b)]\itemsep3pt
\item[(a)] $\lim_{n\to\infty}\Pi_{\tau_{n}(i)}=0$ a.s. and $\Erw_{i} J_{i}(\log^{+}|B_{1}^{i}|)<\infty$. In this case, $Q_{i}=\Prob_{i}(Z_{\infty}\in\cdot)$ and $\Psi_{1:n}(Z_{0})\to Z_{\infty}$ in $\Prob_{i}$-probability.
\item[(b)] $\Prob_{i}(|A_{1}^{i}|=1)=1$, \eqref{pi-degenerate} and one of the following two conditions hold:
\begin{description}[(b.2)]\itemsep2pt
\item[(b.1)] $\Prob_{i}(\wh{\tau}(i)\in\cdot)$ is aperiodic.
\item[(b.2)] $\Prob_{i}(\wh{\tau}(i)\in\cdot)$ is 2-periodic and the weak limit of $\Pi_{2n}(Z_{0}-c_{M_{2n}})$ under $\Prob_{i}$ exists and is symmetric.
\end{description}
\item[(c)] $\limsup_{n\to\infty}|\Pi_{\tau_{n}(i)}|=\infty$ a.s., \eqref{pi-degenerate} and one of the following two conditions hold:
\begin{description}[(c.2)]\itemsep2pt
\item[(c.1)] $\Pi_{n}\xrightarrow{\Prob_{i}} 0$. In this case, $Q_{i}=\delta_{c_{i}}$ and $\Psi_{1:n}(Z_{0})\xrightarrow{\Prob_{i}}c_{i}$.
\item[(c.2)] $\limsup_{n\to\infty}\Prob_{i}(|\Pi_{n}|>a)>0$ for some $a>0$, $c_{j}=c$ for all $j\in\cS$ and $Z_{0}=c$ $\Prob_{i}$-a.s. for some $c\in\R$. In this case, $Q_{i}=\delta_{c}$ and $\Psi_{1:n}(Z_{0})=c$ $\Prob_{i}$-a.s. for all $n\ge 0$.
\end{description}
\end{description}
Moreover, any of (a), (b.1) and (c.1), if valid for some $i\in\cS$, holds true for all $i$. Finally, if (a) and \eqref{pi-degenerate} both fail, then
$$ |\Psi_{1:n}(Z_0)|\ \xrightarrow{\Prob_\pi}\ \infty $$
for any admissible $Z_{0}$.
\end{Theorem}

\begin{Rem}\rm
A description of the limit laws under Condition (b) can also be given but is postponed because it requires further notation, see Lemma \ref{lem:final lemma case 2}.
\end{Rem}

\begin{Rem}\label{rem:weak convergence}\rm
We emphasize that $\Psi_{1:n}(Z_{0})$ converges in distribution for \emph{all} admissible $Z_{0}$ and under \emph{all} $\Prob_{i}$ if (a), (b.1), or (c.1) is valid for some $i\in\cS$. On the other hand, if (c.2) holds, then the distributional convergence of $\Psi_{1:n}(Z_{0})$ under some $\Prob_{i}$ does not ensure the same under any other $\Prob_{j}$, for $\Prob_{i}(Z_{0}=c)=1$ may be valid only for $i$ in a proper subset of $\cS$. A similar disclaimer applies if (b.2) is valid. Just take $(A_{1},B_{1})=(-1,2c)$ for some $c\in\R$, thus $A_{1}c+B_{1}=c$, and an admissible $Z_{0}$ such that $Z_{0}-c$ is symmetric under some one $\Prob_{i}$ but not symmetric under $\Prob_{j}$, $j\ne  i$. Then $\Psi_{1:n}(Z_{0})=(-1)^{n}(Z_{0}-c)+c$ converges in distribution only under $\Prob_{i}$.
\end{Rem}

\begin{Rem}\label{rem:weak <->a.s. convergence}\rm
As pointed out in Subsection \ref{subsec:fluctuation}, $\Pi_{\tau_{n}(i)}\to 0$ and $\limsup_{n\to\infty}|\Pi_{n}|=\infty$ a.s. may hold together. Assuming this and additionally $\Erw_{i}J_{i}(\log^{+}W^{i})<\infty$ for all $i\in\cS$, which obviously implies $\Erw_{i}J_{i}(\log^{+}|B^{i}|)<\infty$ for all $i\in\cS$, we infer $\Psi_{1:n}(0)\to Z_{\infty}$ a.s. by Theorem \ref{thm:a.s. convergence} and $\Psi_{1:n}(Z_{0})\xrightarrow{\Prob_{i}} Z_{\infty}$ for any admissible $Z_{0}$ and any $i$ by Theorem \ref{thm:weak convergence}. On the other hand, Theorem \ref{thm:weak convergence} also asserts the existence of an admissible $Z_{0}$ such that $\Psi_{1:n}(Z_{0})$ does not converge a.s.
\end{Rem}

\begin{Rem}\label{rem:finite S}\rm
If the driving chain has finite state space $\cS$ and \eqref{pi-degenerate} is ruled out, then the almost sure convergence and the stochastic convergence of $\Psi_{1:n}(Z_{0})$ under any $\Prob_{i}$ (and thus under $\Prob_{\pi}$) are equivalent for admissible $Z_{0}$. To see this, we first note that $\Psi_{1:n}(Z_{0})\xrightarrow{\Prob_{\pi}}Z_{\infty}$ implies $\Pi_{\tau_{n}(i)}\to 0$ a.s. and $\Erw_{i}J_{i}(\log^{+}|B^{i}|)<\infty$. Now use Theorem \ref{thm:G/M theorem} for the iid linear maps $\Psi_{1}^{i},\Psi_{2}^{i},\ldots$ defined in Subsection \ref{subsec:return times} to infer $\Psi_{1:n}^{i}(Z_{0})\to Z_{\infty}$ a.s. for all $i\in\cS$. Finally, putting $N_{i}(n):=\sup_{k\ge 0}:\tau_{k}(i)\le n\}$ for $i\in\cS$, it follows that
$$ \left|\Psi_{1:n}(Z_{0})-Z_{\infty}\right|\ \le\ \max_{i\in\cS}\left|\Psi_{1:N_{i}(n)}(Z_{0})-Z_{\infty}\right|\ \to\ 0\quad\text{a.s.} $$
because $\cS$ is finite, $\lim_{n\to\infty}\min_{i\in\cS}N_{i}(n)=\infty$ a.s. and $\min_{i\in\cS}|N_{i}(n)-n|=0$ for all $n$. As a by-product, the equivalence of \ref{thm:a.s. convergence}(b) and \ref{thm:weak convergence}(a) is obtained.
\end{Rem}

\begin{Rem}\label{rem:MRW property}\rm
If $\Prob_{\pi}(A=1)=1$ and \eqref{degenerate} fails, then $\Psi_{1:n}(0)=\sum_{k=1}^{n}B_{k}$ forms a nontrivial, i.e. not null-homologous MRW. By the last assertion of Theorem \ref{thm:weak convergence}, we infer that any such MRW converges to $\infty$ in $\Prob_{\pi}$-probability, but this will actually be needed for the proof of that assertion and therefore be proved independently in Lemma \ref{lem:concentration MRW}.
\end{Rem}

\subsection{The fixed-point property}\label{subsec:FP property}

In the classical case of iid $(A_{n},B_{n})$, any distributional limit $Q$, say, of $\Psi_{1:n}(Z_{0})$ must satisfy the stochastic fixed point equation (SFPE) $R\eqdist\Psi(R')=AR'+B$ as already mentioned (see \eqref{SFPE iid}). Here $R$ and $R'$ have law $Q$ and $(A,B)$ is independent of $R'$. The fixed-point property refers to $\Psi$ interpreted as a map on the set $\cP(\R)$ of probability distributions on $(\R,\cB(\R))$ which maps any $Q$ from this set to the law of $AR'+B$ with $R'\eqdist Q$ independent of $(A,B)$. Obviously, \eqref{SFPE iid} may then also be stated as $\Psi Q=Q$.

\vspace{.1cm}
In the Markov-modulated case considered here, the fixed-point property of the distributional limits of $\Psi_{1:n}(Z_{0})$ requires adjustment even in the most comfortable situation when $\Psi_{1:n}(Z_{0})$ converges a.s. to some $Z_{\infty}$. Since $\Psi_{2:n}(Z_{0})$ then clearly converges a.s. to some $Z_{\infty}'$ having the same law as $Z_{\infty}$ under $\Prob_{\pi}$, the recursive relation
$$ \Psi_{1:n}(Z_{0})\ =\ \Psi_{1}(\Psi_{2:n}(Z_{0})) $$
in combination with the continuity of $\Psi_{1}$ leads to the conclusion that
\begin{equation}\label{SFPE false 2}
Z_{\infty}\ =\ \Psi_{1}(Z_{\infty}')\ =\ A_{1}Z_{\infty}'+B_{1}\quad\text{a.s.}
\end{equation}
which in turn suggests validity of an even stronger relation than \eqref{SFPE iid} (at least under $\Prob_{\pi}$), for $\eqdist$ appears to be replaced with an equality in terms of random variables. However, the reader should observe that $Z_{\infty}'$, though a copy of $Z_{\infty}$, is \emph{not} independent of $(A_{1},B_{1})$ whence \eqref{SFPE false 2} becomes actually ambiguous when stated as a distributional relation.

\vspace{.1cm}
The crucial point is that the Markovian environment must be taken into account with the result that the fixed point property refers now to a map that acts on probability kernels $P$ instead of distributions. Let $\cP(\cS,\R)$ denote the set of such kernels from $\cS$ to $\R$ and interpret $\Psi_{1}$ as a map on $\cP(\cS,\R)$ which sends an element $P$ to the conditional law of $A_{1}R'+B_{1}$ given $M_{0}$ when $R'$ and $(A_{1},B_{1})$ are conditionally independent given $(M_{0},M_{1})$ and $\cL(R'|M_{0}=i,M_{1}=j)$, the conditional law of $R'$ given $M_{0}=i,M_{1}=j$, equals $P(j,\cdot)$ for any $i,j\in\cS$. If
\begin{equation}\label{eq:SFPE for Psi}
\Psi_{1}P(i,\cdot)\ =\ P(i,\cdot)\quad\text{for all }i\in\cS, 
\end{equation}
then $P$ is called a solution to this equation, or a fixed point of $\Psi_{1}$, and it is now readily seen that, if $P(i,\cdot)$ equals the conditional of $Z_{\infty}$ given $M_{0}=i$, then $P$ is indeed a fixed point of $\Psi_{1}$.

\vspace{.1cm}
Given such $P$ of $\Psi_{1}$, let $R$ be a random variable with conditional law $P(i,\cdot)$ given $M_{0}=i$. Then
\begin{align*}
P(i,E)\ &=\ \Prob_{i}(R\in E)\ =\ \Erw_{i}\1_{E}(\Psi_{1}(R'))\\
&=\ \sum_{j\in\cS}p_{ij}\int\Prob_{i}(\Psi_{1}(r)\in E|M_{1}=j)\ P(j,dr)
\end{align*}
for any measurable $E\subset\R$ and $i\in\cS$, which in terms of random variables may be stated as
\begin{equation}\label{Markovian SFPE}
R^{i}\ \eqdist\ \Psi_{1}(R^{M_{1}})\ =\ \sum_{j\in\cS}\1_{\{M_{1}=j\}}\,\Psi_{1}(R^{j})\quad\text{under }\Prob_{i}
\end{equation}
for all $i$, where $R^{i}$ has distribution $P(i,\cdot)$ under \emph{any} $\Prob_{j}$, $j\in\cS$, and is independent of all other occurring random variables. As one can easily derive by iteration of \eqref{Markovian SFPE}, $R^{i}$ is also a solution to the ordinary SFPE
\begin{equation}\label{SFPE2}
R^{i}\ \eqdist\ A_{1}^{i}R^{i}+B_{1}^{i}
\end{equation}
for each $i\in\cS$. In other words, $P$ forms a solution to the system of ordinary SFPE given by \eqref{SFPE2} for all $i\in\cS$. On the other hand, this is a weaker property than \eqref{Markovian SFPE}. In fact, case (C3) of the subsequent theorem provides an instance where \eqref{SFPE2} for all $i\in\cS$ is solved by any kernel $P$, whereas this is not the case for \eqref{Markovian SFPE} which puts an additional constraint on the relation between the $P(i,\cdot)$ for $i\in\cS$.

\begin{Theorem}\label{thm:fixed-point property}
Under the stated assumptions, suppose that a fixed point $P\in\cP(\cS,\R)$ of $\Psi_{1}$ exists and that $R$ denotes a random variable with conditional law $P(i,\cdot)$ given $M_{0}=i$. Then one of the following cases occur:
\begin{description}[(C3)]\itemsep3pt
\item[(C1)] $\Pi_{\tau_{n}(i)}\to 0$ a.s. and $\Erw_{i} J_{i}(\log^{+}|B_{1}^{i}|)<\infty$ all $i\in\cS$. In this case, $P$ is the unique fixed point of $\Psi_{1}$ and equals the law of the perpetuity $Z_{\infty}=\sum_{n\ge 1}\Pi_{n-1}B_{n}$ under $\Prob_{i}$.
\item[(C2)] $\Prob_{i}(A_{1}^{i}=1)<\Prob_{i}(|A_{1}^{i}|=1)=1$ for all $i\in\cS$. In this case, \eqref{pi-degenerate} holds for a unique sequence $(c_{i})_{i\in\cS}$ and there exists a positive sequence $(a_{i})_{i\in\cS}$ such that $P(i,\cdot)$ equals the law of $a_{i}X+c_{i}$ for a random variable $X$ with symmetric distribution $F$ on $\R$. Equivalently,
$$ R\ =\ a_{M_{0}}X+c_{M_{0}} $$
for a symmetric random variable $X$ independent of $M_{0}$.
\item[(C3)] $\Prob_{i}(A_{1}^{i}=1,B_{1}^{i}=0)=1$ for all $i\in\cS$. In this case, there exist an infinite class $\cC$ of sequences $(c_{i})_{i\in\cS}$ that can be parametrized by the value of $c_{i_{0}}$ for any fixed $i_{0}\in\cS$, a positive sequence $(a_{i})_{i\in\cS}$ and a $\{\pm1\}$-valued sequence $(\sigma_{i})_{i\in\cS}$ such that $P(i,\cdot)$ equals the law of $a_{i}\sigma_{i}X+c_{i}$ for a random variable $X$ with arbitrary distribution $F$ on $\R$. Equivalently, 
$$ R\ =\ a_{M_{0}}\sigma_{M_{0}}X+c_{M_{0}} $$ 
for a random variable $X$ independent of $M_{0}$.
\item[(C4)] $\limsup_{n\to\infty}|\Pi_{\tau_{n}(i)}|=\infty$ a.s.  for all $i\in\cS$ In this case, $\Prob_{i}(A_{1}^{i}c_{i}+B_{1}^{i}=c_{i})=1$ for a unique sequence $(c_{i})_{i\in\cS}$ and $P(i,\cdot)=\delta_{c_{i}}$ equals the unique fixed point of $\Psi_{1}$.
\end{description}
\end{Theorem}

\subsection{Convergence of forward iterations}\label{subsec:convergence forward}

Equation \eqref{eq:forward<->backward dual} suggests to study convergence of the forward iteration $\Psi_{n:1}(Z_{0})$ by looking at the backward dual ${}^{\#}\Psi_{1:n}(Z_{0})$ for any admissible $Z_{0}$ and using the results from Subsection \ref{subsec:3.1}. Unfortunately, this does only work in situations where the latter sequence converges a.s. In general, however, the forward iteration requires its own analysis because
\begin{itemize}
\item $Z_{0}$, if nonconstant, does no longer need to be admissible for the backward dual, and
\item \eqref{eq:forward<->backward dual} generally fails to hold under $\Prob_{i}$ for $i\in\cS$. 
\end{itemize}
Needless to say that, unless $\Psi_{1}(c)=c$ a.s. for some $c\in\R$,
only convergence in distribution occurs because $(M_{n},\Psi_{n:1}(Z_{0}))_{n\ge 0}$ forms a Markov chain.

\vspace{.1cm}
Regarding the degeneracy condition \eqref{pi-degenerate} for the dual $({}^{\#}M_{n},{}^{\#}A_{n+1},{}^{\#}B_{n+1})_{n\ge 0}$, i.e.
\begin{equation*}
\Prob_{\pi}({}^{\#}A_{1}c_{{}^{\#}\!M_{1}}+{}^{\#}B_{1}=c_{{}^{\#}\!M_{0}})\,=\,1\quad\text{for suitable }c_{i}\in\R,\,i\in\cS,
\end{equation*}
it can be restated for the original sequence when using $({}^{\#}M_{0},{}^{\#}M_{1},{}^{\#}A_{1},{}^{\#}B_{1})\eqdist (M_{0},M_{1},A_{1},B_{1})$ under $\Prob_{\pi}$, namely
\begin{equation}\label{dual pi-degenerate}
\Prob_{\pi}(A_{1}c_{M_{0}}+B_{1}=c_{M_{1}})\,=\,1\quad\text{for suitable }c_{i}\in\R,\,i\in\cS.
\end{equation}
By iteration, the latter condition further implies
\begin{equation}\label{dual pi-degenerate iterate}
\Psi_{n:1}(Z_{0})\ =\ c_{M_{n}}\,+\,\Pi_{n}(Z_{0}-c_{M_{0}})\quad\text{a.s.}
\end{equation}

Finally, note that if $({}^{\#}A_{1}^{i},{}^{\#}B_{1}^{i})$ denotes the dual counterpart of $(A_{1}^{i},B_{1}^{i})$, then its law under $\Prob_{i}$ equals
$$ \Prob_{\pi}\left((A_{1}^{i},B_{1}^{i})\in\cdot|M_{\tau(i)}=i\right). $$
The following result is the forward counterpart of Theorem \ref{thm:weak convergence}.

\begin{Theorem}\label{thm:convergence forward}
Suppose \eqref{standing assumption}. Given $i\in\cS$ and an admissible $Z_{0}$, $\Prob_{i}(\Psi_{n:1}(Z_{0})\in\cdot)$ converges weakly to some $Q_{i}$ iff one of the following conditions is fulfilled:
\begin{description}[(b)]\itemsep3pt
\item[(a)] $\lim_{n\to\infty}\Pi_{\tau_{n}(i)}=0$ a.s. and $\Erw_{i} J_{i}(\log^{+}|{}^{\#}B_{1}^{i}|)<\infty$. In this case, $Q_{i}$ does not depend on $i$ and equals $\Prob_{\pi}({}^{\#}Z_{\infty}\in\cdot)$.
\item[(b)] $\Prob_{i}(|A_{1}^{i}|=1)=1$, \eqref{dual pi-degenerate} and one of the following two conditions hold:
\begin{description}[(b.2)]\itemsep2pt
\item[(b.1)] $\Prob_{i}(\wh{\tau}(i)\in\cdot)$ is aperiodic.
\item[(b.2)] $\Prob_{i}(\wh{\tau}(i)\in\cdot)$ is 2-periodic and 
\begin{equation}\label{eq:symmetry condition forward}
\lim_{n\to\infty}\Prob_{i}(c_{M_{2n}}+\Pi_{2n}(Z_{0}-c_{i})\in\cdot)\ =\ \lim_{n\to\infty}\Prob_{i}(c_{M_{2n}}-\Pi_{2n}(Z_{0}-c_{i})\in\cdot).
\end{equation}
\end{description}
In both cases, $Q_{i}$ may vary with $i$.
\item[(c)] $\limsup_{n\to\infty}|\Pi_{\tau_{n}(i)}|=\infty$ a.s., \eqref{pi-degenerate} and one of the following two conditions hold:
\begin{description}[(c.2)]\itemsep2pt
\item[(c.1)] $\Pi_{n}\xrightarrow{\Prob_{i}} 0$.
\item[(c.2)] $\limsup_{n\to\infty}\Prob_{i}(|\Pi_{n}|>a)>0$ for some $a>0$ and $Z_{0}=c_{i}$.
\end{description}
In both cases, $Q_{i}$ does not depend on $i$ and equals $\Prob_{\pi}(c_{M_{0}}\in\cdot)$.
\end{description}
\end{Theorem}

The form of $Q_{i}$ under Condition (b) will be provided by \eqref{eq1:def Q_i under (b.1)}--\eqref{eq:def Q_i under (b.2)} in the proof of Lemma \ref{lem:Case (b) forward iteration}.

\subsection{If condition \eqref{standing assumption} fails}

If $\Prob_{\pi}(B=0)=1$, then $\Psi_{n:1}(Z_{0})=\Psi_{1:n}(Z_{0})=\Pi_{n}Z_{0}$ for $n\ge 0$.
Therefore these iterations converge a.s. (to 0) for any admissible $Z_{0}$ iff $\Pi_{n}\to 0$ a.s. (Case (T1') of Proposition \ref{prop:trichotomy embedded RW}), and they converge in distribution under any $\Prob_{i}$ iff $\Pi_{n}\to 0$ a.s. or $\Prob_{i}(|\Pi_{\tau(i)}|=1)=1$ for some/all $i\in\cS$. Under the last assumption (Case (T2') of Proposition \ref{prop:trichotomy embedded RW}), $(\Psi_{n:1}(Z_{0}))_{n\ge 0}$ and $(\Psi_{1:n}(Z_{0}))_{n\ge 0}$ are both regenerative processes under each $\Prob_{i}(\cdot|Z_{0}=z)$, $z\in\R$, with first regeneration epoch $\wh{\tau}(i)$ and weak limits (see Asmussen \cite[Thm. VI.1.2 on p.~170]{Asmussen:03})
\begin{equation}\label{weak limits of iterations}
\frac{1}{\Erw_{i}\wh{\tau}(i)}\,\Erw_{i}\left(\sum_{n=0}^{\wh{\tau}(i)-1}\1_{\{\Psi_{n:1}(z)\in\cdot\}}\right)\quad\text{and}\quad\frac{1}{\Erw_{i}\wh{\tau}(i)}\,\Erw_{i}\left(\sum_{n=0}^{\wh{\tau}(i)-1}\1_{\{\Psi_{1:n}(z)\in\cdot\}}\right),
\end{equation}
respectively, provided that $\Prob_{i}(\wh{\tau}(i)\in\cdot)$ is aperiodic. The limits are generally different and also dependent of $i$. Replacing $z$ with $Z_{0}$, the weak limits of $\Psi_{n:1}(Z_{0})$ and $\Psi_{1:n}(Z_{0})$ for any admissible $Z_{0}$ are obtained. If $\Prob_{i}(\wh{\tau}(i)\in\cdot)$ is 2-periodic, then similar arguments as will be given for the case (b.2) in Theorems \ref{thm:weak convergence} and \ref{thm:convergence forward} show that the weak convergence to the above limits remains valid iff the respective restrictions on $Z_{0}$ stated there hold. We omit further details.

\vspace{.1cm}
If $\Prob_{\pi}(A=0)>0$, let $T:=\inf\{n:A_{n}=0\}$. In this case, the backward iteration $\Psi_{1:n}(Z_{0})$ for arbitrary $Z_{0}$ (admissible or not) equals $\sum_{k=1}^{T}\Pi_{k-1}B_{k}$ a.s. for all $n\ge T$ and is thus a.s. convergent. Regarding the forward iterations $\Psi_{n:1}(Z_{0})$ for admissible $Z_{0}$,
the a.s. convergence of the backward dual ${}^{\#}\Psi_{1:n}(Z_{0})$ in combination with  \eqref{eq:forward<->backward dual} provides us with the weak convergence of $\Prob_{\pi}(\Psi_{n:1}(Z_{0})\in\cdot)$ to $Q_{\pi}:=\Prob_{\pi}\big(\sum_{k=1}^{{}^{\#}T}{}^{\#}\Pi_{k-1}{}^{\#}B_{k}\in\cdot\big)$, where ${}^{\#}T:=\inf\{n:{}^{\#}A_{n}=0\}$. Then, by using a coupling argument, one can also derive that $\Prob_{i}(\Psi_{n:1}(Z_{0})\in\cdot)$ converges to the same law for any $i\in\cS$. Finally, note that $(\Psi_{n:1}(Z_{0}))_{n\ge 0}$ for admissible $Z_{0}$ forms a regenerative process under $\Prob_{\pi}$ with first regeneration epoch $T(i):=\inf\{n:M_{n}=i,A_{n}=0\}$ for any $i$ such that $\Prob_{\pi}(M_{1}=i,A_{1}=0)>0$. As a consequence, $Q_{\pi}$ may alterantively be given as
$$ Q_{\pi}\ =\ \frac{1}{\Erw_{i}T(i)}\Erw_{i}\left(\sum_{n=1}^{T(i)}\1_{\{\Psi_{n:1}(0)\in\cdot)\}}\right). $$

\section{The degeneracy condition \eqref{pi-degenerate}}\label{sec:degeneracy}

In the iid-case, degeneracy occurs when $\Prob(Ac+B=c)=1$ for some $c\in\R$ and takes a side note only to deal with. This is quite different in the presence of a Markovian environment, and this subsection therefore collects some relevant facts about the degeneracy condition \eqref{pi-degenerate} and particularly shows its equivalence with \eqref{degenerate}. To avoid trivialities, we assume $|\cS|>1$ throughout.

\vspace{.1cm}
The first thing to point out about \eqref{degenerate} is the following solidarity lemma.

\begin{Lemma}\label{lem:solidarity}
If $\Prob_{i}(A_{1}^{i}c_{i}+B_{1}^{i}=c_{i})=1$ for some one pair $(i,c_{i})\in\cS\times\R$, then \eqref{degenerate} holds.
\end{Lemma}

The proof is furnished by the following three results, the first of which was shown by Grincevi\v{c}ius \cite[Prop. 1]{Grincev:82}.

\begin{Lemma}\label{lem:Grincevicius}
Given a bivariate sequence $(A_{n},B_{n})_{n\ge 1}$ of iid real-valued random variables, the following assertions are equivalent:
\begin{description}[(b)]\itemsep2pt
\item[(a)] $A_{1}B_{2}+B_{1}=g(A_{1}A_{2})$ a.s. for some measurable function $g$.
\item[(b)] For some $c\in\R$, either $A_{1}c+B_{1}=c$ a.s. or $(A_{1},B_{1})=(1,c)$ a.s.
\end{description}
\end{Lemma}

\begin{Lemma}\label{lem:A_1=1,B_1=c}
If there are constants $b_{i}\in\R$ such that
$$ \Prob_{i}((A_{1}^{i},B_{1}^{i})=(1,b_{i}))\ =\ 1\quad\text{for all }i\in\cS, $$
then $b_{i}=0$ and thus \eqref{degenerate} trivially holds with $c_{i}=0$ for all $i\in\cS$.
\end{Lemma}

\begin{proof}
Fixing any $i\in\cS$, there exist $j\in\cS\backslash\{i\}$ (recall $|\cS|>1$) and $n_{0},n_{1},n_{2}\in\N$ such that
\begin{align*}
\Prob_{i}(\tau(i)=n_{1})\ &\ge\ \Prob_{i}(\tau(i)=n_{1},\tau(j)>n_{1})\ >\ 0
\shortintertext{and}
\Prob_{j}(\tau(j)>n_{0},M_{n_{0}}=i)\ &\ge\ \Prob_{j}(\tau(j)=n_{0}+n_{2},M_{n_{0}}=i)\ >\ 0.
\end{align*}
It follows that $E_{1}:=\{M_{0}=j,\,\tau(j)=n_{0}+n_{2}\}$ and $E_{2}:=\{M_{0}=j,\,M_{n_{0}}=M_{n_{0}+n_{1}}=i,\tau(j)=n_{0}+n_{1}+n_{2}\}$ have positive $\Prob_{j}$-probability.
On $E_{1}$, the proviso provides us with
$$ b_{j}\ =\ B_{1}^{j}\ =\ \sum_{k=1}^{n_{0}+n_{2}}\Pi_{k-1}B_{k}, $$
while on $E_{2}=\{M_{0}=j,\,M_{n_{0}}=M_{n_{0}+n_{1}}=i,\tau(j)=n_{0}+n_{1}+n_{2}\}$, we find
\begin{align*}
b_{j}\ =\ B_{1}^{j}\ &=\ \sum_{k=1}^{n_{0}}\Pi_{k-1}B_{k}\ +\ \Pi_{n_{0}}\underbrace{\sum_{k=n_{0}+1}^{n_{0}+n_{1}}\left(\prod_{l=0}^{k-1}A_{n_{0}+l}\right)B_{k}}_{=b_{i}}\\
&+\ \underbrace{\left(\prod_{l=1}^{n_{1}}A_{n_{0}+l}\right)}_{=1}\Pi_{n_{0}}\sum_{k=1}^{n_{2}}\left(\prod_{l=1}^{k-1}A_{n_{0}+n_{1}+l}\right)B_{n_{0}+n_{1}+k}.
\end{align*}
But the fact that $(A_{n},B_{n})_{n\ge 1}$ is modulated by the chain $(M_{n})_{n\ge 0}$ further entails that
\begin{align*}
&\Prob\left(\sum_{k=1}^{n_{0}}\Pi_{k-1}B_{k}+\Pi_{n_{0}}\sum_{k=1}^{n_{2}}\left(\prod_{l=1}^{k-1}A_{n_{0}+n_{1}+l}\right)B_{n_{0}+n_{1}+k}=b_{j}\Bigg|E_{2}\right)\,=\,\Prob\left(B_{1}^{j}=b_{j}\Big|E_{1}\right)\,=\,1,
\end{align*}
whence we finally conclude $b_{j}=\Pi_{n_{0}}b_{i}+b_{j}$ and thus $b_{i}=0$.\qed
\end{proof}

\begin{Lemma}\label{lem:solidarity A_1^s=1}
If $\Prob_{i}(A_{1}^{i}=1)=1$ for some $i\in\cS$, then this is true for all $i\in\cS$.
\end{Lemma}

\begin{proof}
If $A_{1}^{i}=\Pi_{\tau(i)}=1$ $\Prob_{i}$-a.s., then $\Pi_{\tau_{n}(i)}=\prod_{k=1}^{n}A_{k}^{i}=1$ $\Prob_{i}$-a.s. for all $n\ge 1$. Fix an arbitrary $j\ne i$ (if any), write $\tau_{n}$ for $\tau_{n}(j)$, and define
$$ \tau_{n}^{*}(i)\ :=\ \inf\{k>\tau_{n}:M_{k}=i\}\ =\ \min_{k\ge 1}(\tau_{k}(i)-\tau_{n})^{+}. $$
Then
\begin{align*}
1\ =\ \Pi_{\tau_{2}^{*}(i)}\ =\ A_{1}^{j}\left(\prod_{k=\tau_{1}+1}^{\tau_{2}}A_{k}\right)\left(\prod_{k=\tau_{2}+1}^{\tau_{2}^{*}(i)}A_{k}\right)
\end{align*}
and the factors on the right-hand side are independent under $\Prob_{i}$ with
\begin{align*}
\Prob_{i}\left(A_{1}^{j}\prod_{k=\tau_{2}+1}^{\tau_{2}^{*}(i)}A_{k}\in\cdot\right)\ =\ &\Prob_{i}\left(A_{1}^{j}\prod_{k=\tau_{1}+1}^{\tau_{1}^{*}(i)}A_{k}\in\cdot\right)\ =\ \Prob_{i}(\Pi_{\tau_{1}^{*}}(i)\in\cdot)\ =\ \delta_{1}\\
\text{and}\quad&\Prob_{i}\left(\prod_{k=\tau_{1}+1}^{\tau_{2}}A_{k}\in\cdot\right)\ =\ \Prob_{j}(A_{1}^{j}\in\cdot).
\end{align*}
By combining these facts, we find that 
$$ \Prob_{j}(A_{1}^{j}=1)\ =\ \Prob_{i}\left(A_{1}^{j}\prod_{k=\tau_{2}+1}^{\tau_{2}^{*}(i)}A_{k}=1\right)\ =\ 1 $$
as claimed.\qed
\end{proof}

\begin{proof}[of Lemma \ref{lem:solidarity}]
Assuming $\Prob_{i}(A_{1}^{i}c_{i}+B_{1}^{i}=c_{i})=1$ for some $i\in\cS$ and $c_{i}\in\R$, thus $B_{1}^{i}=f(A_{1}^{i})$ $\Prob_{i}$-a.s. for some measurable $f$, we first note that, for all $n\ge 1$,
$$ \sum_{k=1}^{\tau_{n}(i)}\Pi_{k-1}B_{k}\ =\ \sum_{k=1}^{n}\Pi_{\tau_{k-1}(i)}c_{i}(1-A_{k}^{i})\ =\ c_{i}(1-\Pi_{\tau_{n}(i)})\quad\Prob_{i}\text{-a.s.}, $$
i.e., $\sum_{k=1}^{\tau_{n}(i)}\Pi_{k-1}B_{k}=f(\Pi_{\tau_{n}(i)})$ for a measurable function $f$.
We pick again an arbitrary $j\ne i$ and use the notation from the previous proof including $\tau_{n}$ as shorthand for $\tau_{n}(j)$. Then
\begin{align*}
f(\Pi_{\tau_{3}^{*}(i)})\ =\ B_{1}^{j}\,+\,A_{1}^{j}\left[B_{2}^{j}+A_{2}^{j}B_{3}^{j}\right]\,+\,B^{*}
\end{align*}
where $B^{*}:=$ $\sum_{k=\tau_{3}+1}^{\tau_{3}^{*}(i)}\Pi_{k-1}B_{k}$. The left-hand term is deterministic given $A_{1}^{j}$, $\prod_{k=\tau_{1}(j)+1}^{\tau_3(j)} A_{k}$ and $\prod_{k=\tau_{3}(j)+1}^{\tau^{*}(i)} A_{k}$, and this remains of course true 
when additionally conditioning upon $B_{1}^{j}$ and $B^{*}$. As for the term in square brackets on the right-hand side, we then see that it must be deterministic. As this term is independent of the given random variables except for $\prod_{k=\tau_{1}(j)+1}^{\tau_3(j)} A_k=A_{2}^{j}\, A_{3}^{j}$, we arrive at the conclusion 
$$ B_{2}^{j}\,+\,A_{2}^{j}B_{3}^{j}\ =\ g(A_{2}^{j}\, A_{3}^{j})\quad\Prob_{i}\text{-a.s.} $$
for some measurable function $g$. By an appeal to Lemma \ref{lem:Grincevicius}, either $B^{j}=c_{j}(1-A^{j})$ or $(A^{j},B^{j})=(1,c_{j})$ $\Prob_{j}$-a.s. for some $c_{j}\in\R$. Since $j$ was arbitrary, one of these alternatives must hold for all $j\in\cS$. 

Suppose the second alternative to be true for some $j\in\cS$, for there is nothing left to verify otherwise. Then $\Prob(A^{j}=1)=1$ for all $j\in\cS$ by Lemma \ref{lem:solidarity A_1^s=1} which in turn entails the existence of a sequence $(c_{j})_{j\in\cS}$ such that
$$ \Prob_{j}((A^j,B^j)=(1,c_{j}))\ =\ 1\quad\text{for all }j\in\cS, $$
regardless of which alternative is true for any particular $j$. The proof is now completed by an appeal to Lemma \ref{lem:A_1=1,B_1=c}, giving $c_{j}=0$ for all $j\in\cS$ and thus validity of \eqref{degenerate}.\qed
\end{proof}

\begin{Rem}\label{cor:B_1^s=0}\rm
As a direct outcome of the last argument, $\Prob_{i}(A_{1}^{i}=1,B_{1}^{i}=0)=1$ holds either for all $i\in\cS$ or none.
\end{Rem}

\begin{Prop}\label{prop:degeneracy consequences}
Suppose that \eqref{degenerate} holds and $\Prob_{i}(A_{1}^{i}=1)$ $<1$ for some/all $i\in\cS$. Then the sequence $(c_{i})_{i\in\cS}$ in \eqref{degenerate} is unique. Moreover, \eqref{pi-degenerate} is true and, more generally,
\begin{equation}\label{eq:general homology}
\Psi_{1:n}(c_{M_{n}})\ =\ \Pi_{n}c_{M_{n}}+\sum_{k=1}^{n}\Pi_{k-1}B_{k}\ =\ c_{M_{0}}\quad\text{a.s.}
\end{equation}
for any $n\in\N$.
\end{Prop}

\begin{proof}
(a) If \eqref{degenerate} holds, then $\Prob_{i}(A_{1}^{i}=1)<1$ entails
$$ \Prob_{i}\left(\frac{B_{1}^{i}}{1-A_{1}^{i}}=c_{i}\bigg|A_{1}^{i}<1\right)\ =\ 1, $$
and this forces $c_{i}$ to be unique.

\vspace{.1cm}
(b) Turning to \eqref{pi-degenerate}, pick any $i,j\in\cS$ such that $p_{ij}>0$. Then $\Prob_{i}(E_{n})>0$ for all $n\ge 1$, where
$$ E_{n}\ :=\ \{M_{0}=i,M_{1}=M_{\tau_{n}(i)+1}=j\}. $$
Given $E_{n}$, we have $\Psi_{2:\tau_{n}(i)+1}(c_{j})=c_{j}$ and $\Psi_{1:\tau_{n}(i)}(c_{i})=c_{i}$ by \eqref{degenerate} and therefore
\begin{align*}
\Psi_{1}(c_{j})\ =\ \Psi_{1:\tau_{n}(i)+1}(c_{j})\ &=\ c_{i}\ +\ \left(\Psi_{1:\tau_{n}(i)}(\Psi_{\tau_{n}(i)+1}(c_{j}))-\Psi_{1:\tau_{n}(i)}(c_{i})\right)\\
&=\ c_{i}\ +\ \Pi_{\tau_{n}(i)}(\Psi_{\tau_{n}(i)+1}(c_{j})-c_{i})\quad\text{a.s.}
\end{align*}
or, equivalently,
\begin{equation}\label{eq:crucial identity on E_n}
\Psi_{1}(c_{j})-c_{i}\ =\ \Pi_{\tau_{n}(i)}(\Psi_{\tau_{n}(i)+1}(c_{j})-c_{i})\quad\text{a.s.}
\end{equation}
Moreover, $\Psi_{1}(c_{j})-c_{i},\Psi_{\tau_{n}(i)+1}(c_{j})-c_{i}$ are conditionally iid and the second variable also conditionally independent of $\Pi_{\tau_{n}(i)}$ given $E_{n}$.
Now observe that, as another consequence of $\Prob_{i}(A_{1}^{i}=1)<1$,
\begin{align*}
\Prob_{i}(\Pi_{\tau_{k}(i)}=1,\,1\le k\le n|E_{n})\ &=\ \Prob_{i}(A_{k}^{i}=1,\,1\le k\le n|M_{1}=j)\\
&=\ \Prob_{i}(A_{1}^{i}=1|M_{1}=j)\,\Prob_{i}(A_{1}^{i}=1)^{n-1}\ <\ 1
\end{align*}
for all $n\ge 2$ and so $\Prob_{i}(\Pi_{\tau_{n}(i)}=1|E_{n})<1$ for $n=1$ or $n=2$. Going back to \eqref{eq:crucial identity on E_n} for such $n$, we arrive at the conclusion
\begin{align*}
1\,=\,\Prob_{i}(\Psi_{1}(c_{j})-c_{i}=0|E_{n})\,=\,\Prob_{i}(\Psi_{1}(c_{j})-c_{i}=0|M_{1}=j)\,=\,\Prob_{i}(A_{1}c_{j}+B_{1}=c_{i}|M_{1}=j)
\end{align*}
and thus \eqref{pi-degenerate}. Relation \eqref{eq:general homology} then follows by iteration.\qed
\end{proof}

The case $\Prob_{i}(A_{1}^{i}=1)=1$ for all $i\in\cS$ will be treated in the next proposition and bears some differences. Put $\R^{*}:=\R\backslash\{0\}$.

\begin{Prop}\label{prop2:degeneracy consequences}
Assuming \eqref{degenerate} and $\Prob_{i}(A_{1}^{i}=1)=1$ for all $i\in\cS$, the following assertions hold:
\begin{description}[(b)]\itemsep3pt
\item[(a)] There exist functions $f_{A},f_{B}:\cS^{2}\to\R^{*}\times\R$ such that
$$ (A_{1},B_{1})\ =\ (f_{A}(M_{0},M_{1}),f_{B}(M_{0},M_{1}))\quad\text{a.s.} $$
\item[(b)] There exists a family of nonconstant affine linear functions $(\Phi_{ij})_{i,j\in\cS}$, such that $\Phi_{ij}=\Phi_{ji}^{-1}$ for all $i,j\in\cS$, thus $\Phi_{ii}(x)=x$ for all $x\in\R$, and
\begin{align*}
\Prob_{i}(\Psi_{1:n}=\Phi_{ij}|M_{n}=j)\ =\ 1
\end{align*}
whenever $\Prob_{i}(M_{n}=j)>0$, in particular $\Psi_{n}=\Phi_{M_{n-1}M_{n}}$ a.s. for all $n\ge 1$.
\item[(c)] For each $(j,c)\in\cS\times\R$, there exists a sequence $(c_{i})_{i\in\cS}$ such that \eqref{pi-degenerate} and \eqref{eq:general homology} are valid.
\end{description}
\end{Prop}

\begin{proof}
If \eqref{degenerate} holds, then $1=\Prob_{i}(A_{1}^{i}=1)=\Prob_{i}(A_{1}^{i}=1,B_{1}^{i}=0)$ for all $i\in\cS$ implies that $(c_{i})_{i\in\cS}$ can be chosen arbitrarily.

\vspace{.1cm}
(a) Pick any $i,j\in\cS$ with $p_{ij}>0$ and define $E=E_{1}$ as in part (b) of the previous proof. Then $\Psi_{1:\tau(i)}(x)=\Psi_{2:\tau(i)+1}(x)=x$ a.s. on $E$ implies
\begin{align*}
A_{1}x+B_{1}\ &=\ \Psi_{1}(x)\ =\ \Psi_{1:\tau(i)+1}(x)\ =\ \Psi_{\tau(i)+1}(x)\ =\ A_{\tau(i)+1}x+B_{\tau(i)+1}
\end{align*}
a.s. on $E$ for all $x\in\R$. On the other hand, $(A_{1},B_{1})$ and $(A_{\tau(i)+1},B_{\tau(i)+1})$ are conditionally iid with conditional laws depending only on $i,j$. This clearly implies the assertion.

\vspace{.1cm}
(b) For any $i,j\in\cS^{2}$, we can fix a path $j\to j_{1}\to...\to j_{n-1}\to i$ of minimal length such that $\Prob_{j}(M_{1}=j_{1},...,M_{n-1}=j_{n-1},M_{n}=i)>0$. Conditioned on this event, the map $\Psi_{1:n}$ is deterministic by (a) and denoted $\Phi_{ji}$. For any further path $i\to i_{1}\to...\to i_{m-1}\to j$ of positive probability it then follows that $\tau_{k}(i)=m+n$ for some $k\in\N$ on
\begin{align}\label{eq:def of event E}
E\ :=\ \{M_{0}=i,\,M_{k}=i_{k},1\le k<m,\,M_{m}=j,\,M_{m+l}=j_{l},1\le l<n,M_{m+n}=i\}
\end{align}
and thereupon
$$ x\ =\ \Psi_{1:m+n}(x)\ =\ \Psi_{1:m}\circ\Psi_{m+1:m+n}(x)\ =\ \Psi_{1:m}\circ\Phi_{ji}(x) $$
a.s. on $E$, i.e. $\Psi_{1:m}=\Phi_{ji}^{-1}$. Moreover, the maps $\Psi_{1:m}$ are all identical when conditioned upon a path of arbitrary length $m$ from $M_{0}=i$ to $M_{m}=j$, giving $\Psi_{1:m}=\Phi_{ij}$.

\vspace{.1cm}
(c) Fix any $(j,c)\in\cS\times\R$ and put $c_{i}:=\Phi_{ij}(c)$ for $i\in\cS$. Then $c_{M_{0}}=\Phi_{M_{0}j}$ in combination with
$$ A_{1}+c_{M_{1}}\ =\ \Phi_{M_{0}M_{1}}(c_{M_{1}})\ =\ \Phi_{M_{0}M_{1}}(\Phi_{M_{1}j}(c))\ =\ \Phi_{M_{0}j}(c)\quad\text{a.s.} $$
shows  \eqref{pi-degenerate} and then \eqref{eq:general homology} again upon iteration.\qed
\end{proof}

\begin{Lemma}\label{lem:B_1^s>0}
Suppose $\Prob_{\pi}(A=0)=0$. Then the following assertions are true:
\begin{description}[(b)]\itemsep3pt
\item[(a)] $\Prob_{i}(B_{1}^{i}=0)=1$ for some $i\in\cS$ implies \eqref{pi-degenerate}.
\item[(b)] $\Prob_{i}(B_{1}^{i}=0)=1$ for all $i\in\cS$ implies $\Prob_{\pi}(B=0)=1$ or $\Prob_{i}(A_{1}^{i}=1)=1$ for all $i\in\cS$.
\end{description}
\end{Lemma}

\begin{proof}
(a) First note that $\Prob_{i}(B_{1}^{i}=0)=1$ for some $i\in\cS$ entails $\Prob_{i}(A_{1}^{i}c_{i}+B_{1}^{i}=c_{i})=1$. Now \eqref{degenerate} follows by Lemma \ref{lem:solidarity} and then \eqref{pi-degenerate} by an appeal to the previous two propositions.

\vspace{.1cm}
(b) By assumption, we have
$$ \sum_{l=1}^{\tau_{n}(i)}\Pi_{l-1}B_{l}\ =\ \sum_{l=1}^{n}\Pi_{\tau_{l-1}(i)}B_{l}^{i}\ =\ 0\quad\Prob_{i}\text{-a.s.} $$
for all $n\ge 0$ and $i\in\cS$. If $\Prob_{\pi}(B=0)<1$ and $\Prob_{i}(A_{1}^{i}=1)<1$ for all $i\in\cS$, we can find $j,k\in\cS$ such that $p_{kj}>0$ and
$$ \Prob_{k}(B_{1}=0|M_{1}=j)<1 $$
and also $n_{0},n_{1}\in\N$ such that
$$ \Prob_{j}(\tau(k)=n_{0})>0\quad\text{and}\quad\Prob_{k}(\tau(k)=n_{1},\Pi_{n_{1}}\ne 1)>0. $$
Now, if $\Prob_{i}(B_{1}^{i}=0)=1$ for all $i\in\cS$, we have on $E_{1}:=\{M_{0}=j,\tau(k)=n_{0},M_{n_{0}+1}=j\}$ that
\begin{align*}
0\ &=\ \sum_{l=1}^{n_{0}}\Pi_{l-1}B_{l}\,+\,\Pi_{n_{0}}B_{n_{0}+1}\quad\text{a.s.},
\end{align*}
while on $E_{2}=\{M_{0}=M_{n_{0}+n_{1}+1}=j,\tau(k)=n_{0},\tau_{2}(k)=n_{0}+n_{1}\}$
\begin{align*}
0\ &=\ \sum_{l=1}^{n_{0}+n_{1}+1}\Pi_{l-1}B_{l}\ =\ \sum_{l=1}^{n_{0}}\Pi_{l-1}B_{l}\,+\,\Pi_{n_{0}}B_{2}^{k}\,+\,\Pi_{n_{0}}A_{2}^{k}B_{n_{0}+n_{1}+1}\\
&=\ \sum_{l=1}^{n_{0}}\Pi_{l-1}B_{l}\,+\,\Pi_{n_{0}}A_{2}^{k}B_{n_{0}+n_{1}+1}\quad\text{a.s.}
\end{align*}
must hold. Consequently,
\begin{align*}
1\ &=\ \Prob\left(\sum_{l=1}^{n_{0}}\Pi_{l-1}B_{l}\,+\,\Pi_{n_{0}}A_{2}^{k}B_{n_{0}+n_{1}+1}=0\Bigg|E_{2}\right)\\
&=\ \Prob\left(\sum_{l=1}^{n_{0}}\Pi_{l-1}B_{l}\,+\,\Pi_{n_{0}}B_{n_{0}+1}=0\Bigg|E_{1}\right)\\
&=\ \Prob\left(\sum_{l=1}^{n_{0}}\Pi_{l-1}B_{l}\,+\,\Pi_{n_{0}}B_{n_{0}+n_{1}+1}=0\Bigg|E_{2}\right),
\end{align*}
where the last equality follows because the conditional law of $\sum_{l=1}^{n_{0}}\Pi_{l-1}B_{l}\,+\,\Pi_{n_{0}}B_{n_{0}+1}$ given $E_{1}$ coincides with the conditional law of $\sum_{l=1}^{n_{0}}\Pi_{l-1}B_{l}\,+\,\Pi_{n_{0}}B_{n_{0}+n_{1}+1}$ given $E_{2}$, due to the Markov-modulated structure. We thus arrive at
$$ \Prob\left(\Pi_{n_{0}}B_{n_{0}+n_{1}+1}(1-A_{2}^{k})=0|E_{2}\right)\ =\ 1 $$
and thereupon, due to conditional independence, at the conclusion that
$$ \Prob(B_{n_{0}+n_{1}+1}=0|E_{2})\ =\ \Prob_{k}(B_{1}=0)\ =\ 1\quad\text{or}\quad\Prob(A_{2}^{k}=1|E_{2})\ =\ \Prob_{k}(A_{1}^{k}=1)\ =\ 1 $$
which is impossible by construction.\qed
\end{proof}

\begin{Exa}\label{exa:B_1^s>0}\rm
Here is an example where $\Prob_{i}(B^{i}=0)=1$ holds for some, but not all $i\in\cS$. Suppose that $\cS=\{0,1,2,3\}$, $0<p_{01}=1-p_{02}<1$, $p_{23}=p_{30}=p_{10}=1$, $\Prob_{\pi}(B=1)=1$, and
\begin{align*}
\Prob_{0}(A_{1}=-1,A_{2}=1|M_{1}=1)\ =\ \Prob_{0}(A_{1}=-3/2,A_{2}=-1/3,A_{3}=1|M_{1}=2)\ =\ 1.
\end{align*}
Then one can easily check that $B_{1}^{0}=\sum_{k=1}^{\tau(0)}\Pi_{k-1}=0$ $\Prob_{0}$-a.s., whereas $\Prob_{i}(B_{1}^{i}=0)<1$ for any other $i\in\cS$.
\end{Exa}

We will need further information on $(\Pi_{n})_{n\ge 0}$ in the case when $\Prob_{i}(A_{1}^{i}=1)=1$ for all $i\in\cS$. Let $\sign(x)$ denote the sign of $x$, write
$$ \Pi_{n}\ =\ \sign(\Pi_{n})\,|\Pi_{n}|\ =\ \prod_{k=1}^{n}\sign(A_{k})\,|A_{k}|, $$ 
and observe that $(\log|\Pi_{n}|)_{n\ge 0}$ is null-homologous, that is
$$ \log|A_{n}|\ =\ \log|f_{A}(M_{n-1},M_{n})|\ =\ g(M_{n})-g(M_{n-1})\quad\text{a.s.} $$
for a suitable function $g:\cS\to\R$. Information on $\sign(\Pi_{n})$ is provided by the next lemma.

\begin{Lemma}\label{lem:sign Pi_n}
Assuming $\Prob_{i}(A_{1}^{i}=1)=1$ for all $i\in\cS$, there exists a sequence $(\sigma_{j})_{j\in\cS}$ in $\{\pm 1\}$ such that, for all $n\ge 1$,
\begin{equation}\label{eq:sign Pi_n}
\sign(\Pi_{n})\ =\ \sigma_{M_{n}}/\sigma_{M_{0}}\quad\text{a.s.}
\end{equation}
and therefore
\begin{equation}\label{eq:Pi_n null-homologous}
\Pi_{n}\ =\ \sigma_{M_{n}}\,e^{g(M_{n})-g(M_{0})}/\sigma_{M_{0}}\quad\text{a.s.}
\end{equation}
\end{Lemma}

Plainly, $\sigma_{i}\sigma_{j}=\sigma_{j}/\sigma_{i}$ for all $i,j\in\cS$. We have chosen the ratio form because of its mnemonic appeal in connection with null-homology.

\begin{proof}
The following argument is very similar to the one used in the proof of Proposition \ref{prop2:degeneracy consequences}(b): For any $i,j\in\cS$, we fix a path $j\to j_{1}\to...\to j_{n-1}\to i$ of minimal length with $\Prob_{j}(M_{1}=j_{1},...,M_{n-1}=j_{n-1},M_{n}=i)>0$. Conditioned on this event, $\sign(\Pi_{n})$ is deterministic and denoted $\sigma^{*}(i,j)$. For any further path $i\to i_{1}\to...\to i_{m-1}\to j$ of positive probability it then follows that $\tau_{k}(s)=m+n$ for some $k\in\N$ and thus $1=\sign(\Pi_{m+n})=\sign(\Pi_{n})\sigma^{*}(i,j)$ on the event $E$ as defined by \eqref{eq:def of event E}. Hence, $\sign(\Pi_{n})=\sigma^{*}(i,j)$ on $E$, regardless of the particular choice of $i_{1},...,i_{n-1}$. We conclude that $\sign(\Pi_{n})$ given $M_{0},...,M_{n}$ a.s. depends only on the endpoints $M_{0},M_{n}$ for any $n\ge 1$, hence
\begin{equation*}
\sign(\Pi_{n})\ =\ \sigma^{*}(M_{0},M_{n})\quad\text{a.s.}
\end{equation*}
One can easily verify that
\begin{align*}
&\sigma^{*}(i,i)=1,\\
&\sigma^{*}(i,j)=\sigma^{*}(j,i)\\
\text{and}\quad&\sigma^{*}(i,j)=\sigma^{*}(i,k)\sigma^{*}(k,j)
\end{align*}
for all $i,j,k\in\cS$. But this implies that $\sigma^{*}(i,j)=\sigma_{j}/\sigma_{i}$ and thus \eqref{eq:sign Pi_n} when defining $\sigma_{i}:=\sigma(i_{0},i)$ for any fixed element $i_{0}\in\cS$.\qed
\end{proof}

\section{Proof of Theorem \ref{thm:a.s. convergence}}\label{sec:proof 1st thm}

For the equivalence of (a)--(e), we must only show ``(c)$\RA$(b)'' and ``(b)$\RA$(e)'' because
the implications ``(e)$\RA$(a)$\RA$(d)$\RA$(c)'' are trivial. For $i\in\cS$ and $n\ge 1$, we define 
$$ W_{n}^{i}\ :=\ \max_{\tau_{n-1}(i)<k\le\tau_{n}(i)}\left(\prod_{l=1}^{k-1}A_{\tau_{n-1}(i)+l}\right)B_{k} $$ 
(with $\tau_{0}(i):=0$). Under each $\Prob_{j}$, the $W_{n}^{i}$, $n\ge 2$, are iid and independent of $W_{1}^{i}=W^{i}$ with common law $\Prob_{i}(W^{i}\in\cdot)$.

\vspace{.1cm}
``(c)$\RA$(b)'' Suppose first that $\Pi_{\tau_{n}(i)}$ does not converge to 0 a.s. for some/all $i\in\cS$. By Proposition \ref{prop:trichotomy embedded RW}, this implies $\limsup_{n\to\infty}|\Pi_{\tau_{n}(i)}|=\infty$ a.s., for $\Prob_{i}(|\Pi_{\tau(i)}|=1)<1$ is assumed. Recalling \eqref{standing assumption}, we may pick $i$ such that $\Prob_{i}(B=0)<1$. Note that, for any $n\ge 1$ and $j\in\cS$,  $\Pi_{\tau_{n}(i)}$ and $B_{\tau_{n}(i)+1}$ are independent under $\Prob_{j}$ with $\Prob_{j}(B_{\tau_{n}(i)+1}\in\cdot)=\Prob_{i}(B\in\cdot)$. But then
$$ \limsup_{n\to\infty}|\Pi_{n-1}B_{n}|\ \ge\ \limsup_{n\to\infty}|\Pi_{\tau_{n}(i)}B_{\tau_{n}(i)+1}|\ =\ \infty\quad\text{a.s.} $$
which contradicts the second assertion of (c).

\vspace{.1cm}
If $\Erw_{i}J_{i}(\log^{+}W^{i})=\infty$ and thus a fortiori $\Erw_{i}\log^{+}W^{i}=\infty$ for some $i\in\cS$, then we use Erickson's lemma \cite[Lemma 4]{Erickson:73} (in a slightly more general form also used in \cite[Lemma 5.2]{GolMal:00} and proved as Lemma 7.1 in \cite {AlsBuck:16}) to infer
$$ \limsup_{n\to\infty}\frac{\log^{+}W_{n+1}^{i}}{\sum_{k=1}^{n}(S_{\tau_{k}(i)}-S_{\tau_{k-1}(i)})^{+}}\ =\ \infty\quad\text{a.s.} $$
and thereby
\begin{align*}
\infty\ &=\ \limsup_{n\to\infty}\left[\,\sum_{k=1}^{n}(S_{\tau_{k}(i)}-S_{\tau_{k-1}(i)})^{+} \left(-1+ \frac{\log^{+}(W_{n+1}^{i})}{\sum_{k=1}^{n}(S_{\tau_{k}(i)}-S_{\tau_{k-1}(i)})^{+}}\right)\right]\\
&\le\ \limsup_{n\to\infty}\left[-S_{\tau_{n}(i)}+\log^{+}( W_{n+1}^{i})\right]\\
&=\ \limsup_{n\to\infty}\left[\log^{+}|\Pi_{\tau_{n}(i)}W_{n+1}^{i}|\right]\quad\text{a.s.}
\end{align*}
Hence, $\infty=\limsup_{n\to\infty}|\Pi_{\tau_{n}(i)}W_{n+1}^{i}|=\limsup_{n\to\infty}|\Pi_{n-1}B_{n}|$ a.s. which again contradicts (c).

\vspace{.2cm}
``(b)$\RA$(e)'' Evidently, (e) follows if we can show the stronger assertion
\begin{equation}\label{eq:exponential decay}
\lim_{n\to\infty}e^{cn}\Pi_{n}B_{n+1}\ =\ 0\quad\text{a.s. for some }c>0.
\end{equation}
Fix any $i\in\cS$, put $N(n):=\sup\{k\ge 1:\tau_{k}(i)\le n\}$ and note that $n/N(n)\to\Erw_{i}\tau(i)$ a.s. by the elementary renewal theorem, thus $n/\tau_{N(n)}(i)\to 1$ a.s. The latter entails
$$ e^{cn}\Pi_{n}B_{n+1}\ \le\ e^{cn}\Pi_{\tau_{N(n)}(i)}W_{N(n)+1}^{i}\ \asymp\ e^{c\tau_{N(n)}(i)}\Pi_{\tau_{N(n)}(i)}W_{N(n)+1}^{i}\quad\text{a.s.} $$
as $n\to\infty$, where $f(n)\asymp g(n)$ means that $0<\liminf_{n\to\infty}\frac{f(n)}{g(n)}\le\limsup_{n\to\infty}\frac{f(n)}{g(n)}<\infty$.
Since $(e^{c\tau_{n}(i)}\Pi_{\tau_{n}(i)}W_{n+1}^{i})_{n\ge 0}$ forms a subsequence of $(e^{cn}\Pi_{n}B_{n+1})_{n\ge 0}$, we see that \eqref{eq:exponential decay} is equivalent to
$$ \lim_{n\to\infty}e^{c\tau_{n}(i)}\Pi_{\tau_{n}(i)}W_{n+1}^{i}\ =\ 0\quad\text{a.s. for some }c>0 $$
which, after a logarithmic transformation, takes the form
\begin{equation}\label{eq:exponential decay additive}
\lim_{n\to\infty}\left(S_{\tau_{n}(i)}-c\tau_{n}(i)-\log W_{n+1}^{i}\right)\ =\ \infty\quad\text{a.s. for some }c>0.
\end{equation}

\vspace{.1cm}
By assumption, $S_{\tau_{n}(i)}\to\infty$ a.s. so that either 
\begin{align*}
&\Erw_{i}S_{\tau(i)}\in (0,\infty)\quad\text{and}\quad\frac{S_{\tau_{n}(i)}}{\tau_{n}(i)}\ \to\ \Erw_{i}S_{\tau(i)}\quad\text{a.s.},
\shortintertext{or}
&\Erw_{i}|S_{\tau(i)}|=\infty\quad\text{and}\quad\frac{S_{\tau_{n}(i)}}{\tau_{n}(i)}\ \to\ \infty\quad\text{a.s.}
\end{align*}
For the last statement, we refer to Kesten's trichotomy \cite{Kesten:70} (see also \cite[Thm. 4 on p.~156]{Chow+Teicher:97}) in the case when $\Erw_{i}S_{\tau(i)}^{+}=\Erw_{i}S_{\tau(i)}^{-}=\infty$. Put $c:=\Erw_{i}S_{\tau(i)}/2$ in the first case, and $c=1$ in the second case. Then we infer
$$ \lim_{n\to\infty}\frac{\tau_{n}(i)}{S_{\tau_{n}(i)}}\ \le\ \frac{1}{2c}\quad\text{a.s.} $$
Furthermore, it has been shown in \cite{AlsBuck:16} (see the proof of Thm.~5.1, ``(b)$\RA$(a)'') that
\begin{equation*}
\limsup_{n\to\infty}\frac{\log^{+}W_{n+1}^{i}}{S_{\tau_{n}(i)}}\ =\ 0\quad\text{a.s.}
\end{equation*}
By combining these facts, we finally obtain
\begin{align*}
\liminf_{n\to\infty}&\left(S_{\tau_{n}(i)}-c\tau_{n}(i)-\log^{+}W_{n+1}^{i}\right)\\
&=\ \lim_{n\to\infty}S_{\tau_{n}(i)}\left(1-c\,\frac{\tau_{n}(i)}{S_{\tau_{n}(i)}}-\frac{\log^{+}W_{n+1}^{i}}{S_{\tau_{n}(i)}}\right)\\
&\ge\ \lim_{n\to\infty}S_{\tau_{n}(i)}\left(\frac{1}{2}-\frac{\log^{+}W_{n+1}^{i}}{S_{\tau_{n}(i)}}\right)\ =\ \infty\quad\text{a.s.}
\end{align*}
and thus \eqref{eq:exponential decay additive}.

\vspace{.2cm}
Turning to the last assertion of the theorem, suppose that $\Psi_{1:n}(Z_{0})$ converges a.s. to a proper limit for any admissible $Z_{0}$. Then
\begin{equation}\label{eq:pi_n(Z_0)-Pi_n(0)}
\Psi_{1:n}(Z_{0})-\Psi_{1:n}(0)\ =\ \Pi_{n}Z_{0}
\end{equation}
does the same for any admissible nonzero $Z_{0}$ and so either $\Pi_{n}\to 0$ a.s. or $\Prob_{\pi}(A=1)=1$. The proof is completed by excluding the last alternative. But $\Prob_{\pi}(A=1)=1$ entails that $\Psi_{1:n}(0)=\sum_{k=1}^{n}B_{k}$, $n\ge 0$, forms a MRW which, by Proposition \ref{prop:trichotomy MRW} and the subsequent remarks, converges a.s. to a proper random variable iff it is null-homologous with $g\equiv 0$, giving $\Prob_{\pi}(B=0)=1$. But the latter is ruled out by \eqref{standing assumption}.\qed

\section{Proof of Theorem \ref{thm:weak convergence}}\label{sec:proof 2nd thm}

The result will be proved separately for the three possible regimes (T1')--(T3') for the multiplicative RW $(\Pi_{\tau_{n}(i)})_{n\ge 0}$ (which is the same for all $i\in\cS$, see Proposition \ref{prop:trichotomy embedded RW}). More precisely, we will show that \ref{thm:weak convergence}(a) provides the necessary and sufficient condition for distributional convergence of $\Psi_{1:n}(Z_{0})$ under (T1'), while \ref{thm:weak convergence}(b) and \ref{thm:weak convergence}(c) do so under (T2') and (T3'), respectively.

\subsection{The case $\lim_{n\to\infty}\Pi_{\tau_{n}(i)}=0$ a.s.}\label{subsec:cond (a)}

We begin with some preliminary facts. For any fixed $i\in\cS$, Lemma \ref{lem:regenerative process} in the Appendix provides us with the distributional convergence of $\ovl{A}_{N(n)}^{i}$ as well as $\ovl{B}_{N(n)}^{i}$ under $\Prob_{\pi}$, where as before $N(n)=\sup\{k\ge 1:\tau_{k}(i)\le n\}$ and
\begin{align*}
\ovl{A}_{n}^{i}\ &:= |\Pi_{\tau_{n-1}(i)}|^{-1}\max_{\tau_{n-1}(i)<k\le\tau_{n}(i)}|\Pi_{k}|,\\
\ovl{B}_{n}^{i}\ &:=\ |\Pi_{\tau_{n-1}(i)}|^{-1}\max_{\tau_{n-1}(i)<k\le\tau_{n}(i)}\left|\sum_{l=1}^{k}\Pi_{l-1}B_{l}\right|
\end{align*}
for $n\ge 1$. As $\Pi_{\tau_{N(n)}(i)}\to 0$ a.s. by the proviso of this subsection, Slutsky's theorem implies
\begin{align}
&|\Pi_{n}|\ \le\ \left|\Pi_{\tau_{N(n)}(i)}\ovl{A}_{N(n)}^{i}\right|\ \xrightarrow{\Prob_{\pi}}\ 0,\quad\text{thus}\quad |\Pi_{n}|\ \xrightarrow{\Prob_{\pi}}\ 0,\label{eq:Pi_n to 0}
\shortintertext{and}
&\max_{\tau_{N(n)}(i)<k\le\tau_{N(n)}(i)+1}\left|\sum_{l=1}^{k}\Pi_{l-1}B_{l}\right|\ \le\ \left|\Pi_{\tau_{N(n)}(i)}\ovl{B}_{N(n)}^{i}\right|\ \xrightarrow{\Prob_{\pi}}\ 0.\label{eq:B-excursions}
\end{align}

The next lemma uses an approach of Goldie and Maller \cite[Lemma 5.5]{GolMal:00}. 

\begin{Lemma}
Suppose \eqref{standing assumption} and $\Pi_{\tau_{n}(j)}\to 0$ a.s. for all $j\in\cS$. Then $\Erw_{i}J_{i}(\log^{+}|B_{1}^{i}|)=\infty$ for some $i\in\cS$ implies $|\Psi_{1:n}(Z_{0})|\xrightarrow{\Prob_{\pi}}\infty$ for any admissible $Z_{0}$.
\end{Lemma}

\begin{proof}
Since $\Pi_{n}\xrightarrow{\Prob_{\pi}}0$ as seen above, we infer $\Pi_{n}Z_{0}\xrightarrow{\Prob_{\pi}}0$ for any admissible $Z_{0}$. In view of \eqref{eq:pi_n(Z_0)-Pi_n(0)}, it therefore suffices to prove $|\Psi_{1:n}(0)|\xrightarrow{\Prob_{\pi}}\infty$. By contraposition, suppose that, for some $i\in\cS$, $\Psi_{1:n}(0)$ does not converge in $\Prob_{i}$-probability to $\infty$ which means that $\Prob_{i}(\Psi_{1:n_{k}}(0)\in\cdot)$ converges vaguely to a nonzero measure $F$ on $\R$ for suitable $n_{1}<n_{2}<\ldots$ We will verify that $\Erw_{i}J_{i}(\log^{+}|B_{1}^{i}|)<\infty$. 

\vspace{.1cm}
Since $F(\R)\le 1$, we can choose a random variable $Z$, independent of all other occurring random variables, such that $\Prob_{i}(Z\in\cdot,|Z|<\infty)=F$. Then we have
$$ \lim_{k\to\infty}\Prob_{i}(x<\Psi_{1:n_{k}}(0)\le y)\ =\ \Prob_{i}(x<Z\le y) $$
for all $x,y\in\cC_{Z}$, $x\le y$, where $\cC_{Z}:=\{x:\Prob_{i}(Z=x)=0\}$. Note that
$$ \Psi_{1:n_{k}}(0)-\Psi_{1:n_{k}-m}(0)\ =\ \Pi_{n_{k}-m}\Psi_{n_{k}-m+1:n_{k}-m}(0)\ \xrightarrow{\Prob_{i}}\ 0 $$
for any $m\ge 0$ because the $\Psi_{n_{k}-m+1:n_{k}-m}(0)$, $k\ge 1$, are iid and $\Pi_{n_{k}-m}\xrightarrow{\Prob_{\pi}}0$. With this at hand, we infer
\begin{align*}
\Prob_{i}&(x<Z\le y)\ =\ \lim_{k\to\infty}\Prob_{i}(x<\Psi_{1:n_{k}}(0)\le y)\\
&=\ \lim_{k\to\infty}\Prob_{i}(x<A_{1}^{i}\,\Psi_{\tau(i)+1:n_{k}}(0)+B_{1}^{i}\le y)\\
&=\ \sum_{m\ge 1}\int\lim_{k\to\infty}\Prob_{i}(x<a\,\Psi_{m+1:n_{k}}(0)+b\le y)\ \Prob_{i}(A_{1}^{i}\in da,B_{1}^{i}\in db,\tau(i)=m)\\
&=\ \sum_{m\ge 1}\int\lim_{k\to\infty}\Prob_{i}(x<a\,\Psi_{1:n_{k}-m}(0)+b\le y)\ \Prob_{i}(A_{1}^{i}\in da,B_{1}^{i}\in db,\tau(i)=m)\\
&=\ \sum_{m\ge 1}\int\lim_{k\to\infty}\Prob_{i}(x<a\,\Psi_{1:n_{k}}(0)+b\le y)\ \Prob_{i}(A_{1}^{i}\in da,B_{1}^{i}\in db,\tau(i)=m)\\
&=\ \sum_{m\ge 1}\int\Prob_{i}(x<aZ+b\le y)\ \Prob_{i}(A_{1}^{i}\in da,B_{1}^{i}\in db,\tau(i)=m)\\
&=\ \Prob_{i}(x<A_{1}^{i}Z+B_{1}^{i}\le y)
\end{align*}
for all $x,y\in\cC_{Z}$ with $x\le y$, in particular $\Prob_{i}(|Z|<\infty)=\Prob_{i}(|A_{1}^{i}Z+B_{1}^{i}|<\infty)$. Therefore, any proper random variable $Z'$ with distribution $\Prob_{i}(Z\in\cdot||Z|<\infty)$ and independent of $(A_{1}^{i},B_{1}^{i})$ satisfies the SFPE \eqref{SFPE2}, i.e.
$$ Z'\ \eqdist\ A_{1}^{i}Z'+B_{1}^{i}. $$
Finally, $\Erw_{i}J_{i}(\log^{+}|B_{1}^{i}|)<\infty$ now follows by invoking Theorem \ref{thm:fixed points IID}.\qed
\end{proof}

The proof of Theorem \ref{thm:weak convergence} under the proviso of this subsection is now completed by the following lemma.

\begin{Lemma}\label{lem:final lemma case 1}
Suppose \eqref{standing assumption} and $\Pi_{\tau_{n}(j)}\to 0$ a.s. for all $j\in\cS$.
Then the following assertions are equivalent for any $i\in\cS$:
\begin{description}[(b)]\itemsep2pt 
\item[(a)] $\Prob_{i}(\Psi_{1:n}(Z_{0})\in\cdot)\weakly Q_{i}$ for some admissible $Z_{0}$ and $Q_{i}\in\cP(\R)$. 
\item[(b)] $\Erw_{i}J_{i}(\log^{+}|B_{1}^{i}|)<\infty$.
\end{description}
Moreover, if (a), (b) do hold for some $i\in\cS$, then $\Psi_{1:n}(Z_{0})\xrightarrow{\Prob_{\pi}}Z_{\infty}$ for any admissible $Z_{0}$ and thus (a), (b) are true for all $i\in\cS$ with $Q_{i}=\Prob_{i}(Z_{\infty}\in\cdot)$.
\end{Lemma}

\begin{proof}
Since ``(a)$\RA$(b)'' is immediate by the previous lemma, we turn directly to the proof of ``(b)$\RA$(a)''. As argued above, it suffices to consider $\Psi_{1:n}(0)$.

\vspace{.1cm}
If $\Prob_{i}(B_{1}^{i}=0)=1$, then \eqref{pi-degenerate} holds by Lemma \ref{lem:B_1^s>0} and so
$$ \Psi_{1:n}(0)\ =\ c_{M_{0}}-\Pi_{n}c_{M_{n}}\quad\text{a.s.} $$
for all $n\ge 1$. Since $(c_{M_{n}})_{n\ge 0}$ is stationary, $\Pi_{n}\xrightarrow{\Prob_{\pi}}0$ implies $\Pi_{n}c_{M_{n}}\xrightarrow{\Prob_{\pi}}0$ and therefore $\Psi_{1:n}(0)\xrightarrow{\Prob_{\pi}}c_{M_{0}}$.

\vspace{.1cm}
If $\Prob_{i}(B_{1}^{i}=0)<1$, then Theorem \ref{thm:G/M theorem} provides us with the $\Prob_{i}$-a.s. convergence of $\Psi_{1:n}^{i}(0)$ and particularly of $\Psi_{1:N(n)}(0)$ to $Z_{\infty}$.
By \eqref{eq:B-excursions} and
$$ \Psi_{1:n}(0)\ =\ \Psi_{1:N(n)}^{i}(0)\ +\ \sum_{k=\tau_{N(n)}+1}^{n}\Pi_{k-1}B_{k}, $$
we then obtain the $\Psi_{1:n}(0)\xrightarrow{\Prob_{i}}Z_{\infty}$. Finally, the same holds true under $\Prob_{\pi}$ because
$$ \Psi_{1:n}(0)\ =\ \1_{\{\tau(i)<n\}}\left(A_{1}^{i}\Psi_{\tau(i)+1:n}(0)+B_{1}^{i}\right)\ +\ \1_{\{\tau\ge n\}}\Psi_{1:n}(0), $$
$\Psi_{\tau(i)+1:n}(0)$ and $(\tau(i),A_{1}^{i},B_{1}^{i})$ are independent under $\Prob_{\pi}$, and
$$ \Prob_{\pi}(\Psi_{\tau(i)+1:n}(0)\in\cdot)\ =\ \Prob_{i}(\Psi_{1:n}(0)\in\cdot).\qquad\qed $$
\end{proof}

\subsection{The case $\Prob_{i}(|A_{1}^{i}|=1)=1$}\label{subsec:cond (b)}

If $\Prob_{i}(|A_{1}^{i}|=1)=1$ for some/all $i\in\cS$, then the MRW $(M_{n},S_{n})_{n\ge 0}$ with $S_{n}=-\log|\Pi_{n}|$ is null-homologous (see Subsection \ref{subsec:fluctuation}), thus
$$ S_{n}\ =\ g(M_{n})-g(M_{0})\quad\text{a.s.} $$
for all $n\ge 0$ and a function $g:\cS\to\R$. Putting $a_{i}:=e^{g(i)}$ for $i\in\cS$, this yields $|\Pi_{n}|=a_{M_{0}}/a_{M_{n}}$ a.s. for all $n\ge 0$, in particular the tightness of $(\Pi_{n})_{n\ge 0}$ and thus of $(\Pi_{n}Z_{0})_{n\ge 0}$ for any admissible $Z_{0}$. As a consequence (see \eqref{eq:pi_n(Z_0)-Pi_n(0)}), $|\Psi_{1:n}(Z_{0})|\xrightarrow{\Prob_{\pi}}\infty$ iff $|\Psi_{1:n}(0)|\xrightarrow{\Prob_{\pi}}\infty$, a fact to be used Lemma \ref{lem:Psi_1:n(0) to infty in (b)} below.

\vspace{.1cm}
We further point out beforehand that $(A_{n},B_{n})_{n\ge 1}$ is also Markov-modulated with respect to the augmented and still positive recurrent Markov chain $(\wh{M}_{n})_{n\ge 0}$ on $\cS\times\{-1,+1\}$, defined by
$$ \wh{M}_{n}\ :=\ (M_{n},\sign(\Pi_{n})). $$
Let $(\wh{\tau}_{n}(i))_{n\ge 1}$ be the subsequence of $(\tau_{n}(i))_{n\ge 0}$ defined by the successive epochs $k$ at which $M_{k}=i$ and $\Pi_{k}=1$, thus $\wh{\tau}_{n}(i)=\tau_{\rho(n)}(i)$ for a renewal stopping sequence $\rho(1),\rho(2),...$ with increment distribution under $\Prob_{i}$ given by
\begin{align}
\begin{split}\label{eq:rho(1) aperiodic}
\Prob_{i}(\rho(1)=n)\ =\ 
\begin{cases}
\hfill\Prob_{i}(A_{1}^{i}=1),&\text{if }n=1,\\
\Prob_{i}(A_{1}^{i}=-1)^{2}\,\Prob_{i}(A_{1}=1)^{n-2},&\text{if }n\ge 2.
\end{cases}
\end{split}
\end{align}
Since $\Erw_{i}\wh{\tau}(i)=\Erw_{i}\tau(i)\,\Erw_{i}\rho(1)<\infty$ by Wald's identity, we see that the augmented chain is indeed positive recurrent. Furthermore, Lemma \ref{lem:lattice-type whtau(i)} below will show that it is at most 2-periodic and that period 2 occurs iff $\Prob_{i}(A_{1}^{i}=-1)=1$. In the aperiodic case, i.e., when $\wh{\tau}(i)$ is aperiodic for some and then (by solidarity) all $i\in\cS$, the ergodic theorem for Markov chains provides us with
\begin{equation}\label{eq:erdodic thm M_n^*}
\lim_{n\to\infty}\Prob_{i}(M_{n}=j,\sign(\Pi_{n})=\pm 1)\ =\ \frac{1}{2}\,\pi_{j}
\end{equation}
for all $i,j\in\cS$.

\vspace{.1cm}
Defining $(\wh{A}_{n}^{i},\wh{B}_{n}^{i})$ in the same manner as $(A_{n}^{i},B_{n}^{i})$, but for the $\wh{\tau}_{n}(i)$, the assumption $\Prob_{i}(|A_{1}^{i}|=1)=1$ implies $\Prob_{i}(\wh{A}_{1}^{i}=1)=1$. Regarding $\wh{B}_{1}^{i}$, we have:

\begin{Lemma}\label{lem:degeneracy (B2)}
If $\Prob_{i}(|A_{1}^{i}|=1)=1$, then \eqref{pi-degenerate} and $\Prob_{i}(\wh{B}_{1}^{i}=0)=1$ are equivalent assertions.
\end{Lemma}

\begin{proof}
Fix any $i\in\cS$. If \eqref{pi-degenerate} holds, then
$$ \wh{B}_{1}^{i}\ =\ \sum_{k=1}^{\rho(1)}\Pi_{\tau_{k-1}(i)}B_{k}^{i}\ =\ c_{i}\sum_{k=1}^{\rho(1)}\Pi_{\tau_{k-1}(i)}(1-A_{k}^{i})\ =\ 1-\wh{A}_{1}^{i}\ =\ 0\quad\Prob_{i}\text{-a.s.} $$
as claimed.

\vspace{.1cm}
Conversely, if $\Prob_{i}(\wh{B}_{1}^{i}=0)=1$, then the MRW $(\wh{M}_{\tau_{n}(i)},\sum_{k=1}^{n}\Pi_{\tau_{k-1}(i)}B_{k}^{i})_{n\ge 0}$ is trivial under $\Prob_{i}$ along its embedded sequence $(\rho(n))_{n\ge 1}$ where $(\wh{M}_{\tau_{n}(i)})_{n\ge 0}$ returns to state $(i,1)$. By Proposition \ref{prop:trichotomy MRW} and the subsequent remarks, it is therefore null-homologous, i.e.
\begin{align*}
B_{1}^{i}\ &=\ g_{i}(A_{1}^{i})-g_{i}(1)\ =\ (g_{i}(-1)-g_{i}(1))\,\1_{\{A_{1}^{i}=-1\}}\quad\Prob_{i}\text{-a.s.}
\end{align*}
for a suitable $g_{i}:\{\pm 1\}\to\R$. Consequently, $B_{1}^{i}=c_{i}(1-A_{1}^{i})$ $\Prob_{i}$-a.s. with $c_{i}:=(g_{i}(-1)-g_{i}(1))/2$. Now use Lemma \ref{lem:solidarity} and Proposition \ref{prop2:degeneracy consequences} to infer \eqref{degenerate} and then \eqref{pi-degenerate}.\qed
\end{proof}

\begin{Lemma}\label{lem:Psi_1:n(0) to infty in (b)}
Suppose that $\Prob_{i}(|A_{1}^{i}|=1)=1$. Then $|\Psi_{1:n}(Z_{0})|\xrightarrow{\Prob_{\pi}}\infty$ for any admissible $Z_{0}$ iff\eqref{pi-degenerate} fails to hold.
\end{Lemma}

\begin{proof}
As noted above, it suffices to consider $\Psi_{1:n}(0)$. Moreover, $|\Psi_{1:n}(0)|\xrightarrow{\Prob_{\pi}}\infty$ a.s. and
$$ |\Psi_{1:n}(0)|\ \xrightarrow{\Prob_{i}}\ \infty\quad\text{for some }i\in\cS $$
are equivalent assertions. Just note that $\Psi_{1:n}(0)=A_{1}^{i}\Psi_{\tau(i)+1:n}(0)+B_{1}^{i}$ a.s. on $\{\tau(i)<n\}$. Fix now any $i\in\cS$. Then
$$ \Psi_{1:\wh{\tau}_{n}(i)}(0)\ =\ \Psi_{1:\rho(n)}^{i}(0)\ =\ \sum_{k=1}^{n}\wh{B}_{k}^{i},\quad n\ge 0 $$
defines an ordinary RW under $\Prob_{i}$. By Lemma \ref{lem:degeneracy (B2)}, it is nontrivial and therefore satisfying $|\Psi_{1:\rho(n)}^{i}(0)|\xrightarrow{\Prob_{i}}\infty$
iff \eqref{pi-degenerate} fails. So it is enough to argue that $|\Psi_{1:\rho(n)}(0)|\xrightarrow{\Prob_{i}}\infty$ entails $|\Psi_{1:n}(0)|\xrightarrow{\Prob_{i}}\infty$.

\vspace{.1cm}
To this end, observe that, for all $n\ge 0$,
\begin{align}\label{eq:inequality Psi_1:n(0)}
|\Psi_{1:\rho(\wh{N}(n))}^{i}(0)|\,-\,B_{\wh{N}(n)+1}^{*i}\ \le\ |\Psi_{1:n}(0)|\ \le\ |\Psi_{1:\rho(\wh{N}(n))}^{i}(0)|\,+\,B_{\wh{N}(n)+1}^{*i}\quad\Prob_{i}\text{-a.s.},
\end{align}
where $\wh{N}(n):=\sup\{k\ge 0:\tau_{\rho(k)}(i)\le n\}=\sup\{k\ge 0:\wh{\tau}_{k}(i)\le n\}$ for $n\ge 0$ and
$$ B_{n}^{*i}\ :=\ \max_{\wh{\tau}_{n-1}(i)<k\le\wh{\tau}_{n}(i)}\left|\sum_{l=1}^{k}\Pi_{l-1}B_{l}\right|\ =\ \max_{\wh{\tau}_{n-1}(i)<k\le\wh{\tau}_{n}(i)}\left|\sum_{l=1}^{k}\left(\prod_{m=\wh{\tau}_{n-1}(i)+1}^{l-1}A_{m}\right)B_{l}\right| $$
for $n\ge 1$. As $(\wh{\tau}_{n}(i))_{n\ge 0}$ is an integrable subsequence of $(\tau_{n}(i))_{n\ge 0}$ with iid increments under $\Prob_{i}$, Lemma \ref{lem:regenerative process} in the Appendix ensures that $B_{\wh{N}(n)+1}^{*i}$ converges in distribution under $\Prob_{i}$. Furthermore, Lemma \ref{lem:concentration MRW} from there provides us with $|\Psi_{1:\rho(\wh{N}(n))}^{i}(0)|\xrightarrow{\Prob_{i}}\infty$ when noting that
$$ \left(\tau_{\rho(\wh{N}(n)+1)}-n,\,\Psi_{1:\rho(\wh{N}(n))}^{i}(0)\right)_{n\ge 0}\ =\ \left(\tau_{\rho(\wh{N}(n)+1)}-n,\,\sum_{k=1}^{\wh{N}(n)}\wh{B}_{k}^{i}\right)_{n\ge 0} $$
constitutes a MRW which has positive recurrent driving chain and is not null-homologous.
The latter holds because the embedded ordinary RW $(\Psi_{1:\rho(n)}^{i})_{n\ge 0}$ obtained at the return times of the driving chain to 0 is nontrivial as stated above. Using these facts in \eqref{eq:inequality Psi_1:n(0)}, we finally conclude $|\Psi_{1:n}(0)|\xrightarrow{\Prob_{i}}\infty$.\qed
\end{proof}

\begin{Lemma}\label{lem:lattice-type whtau(i)}
Suppose that $\Prob_{i}(|A_{1}^{i}|=1)=1$. Then $\wh{\tau}(i)$ is either aperiodic or 2-periodic under $\Prob_{i}$, where the second alternative occurs iff $\Prob_{i}(A_{1}^{i}=-1)=1$.
\end{Lemma}

\begin{proof}
Let $\wh{\phi},\phi,\phi_{+},\phi_{-}$ denote the Fourier transforms of $\wh{\tau}(i),\tau(i),\Prob_{i}(\tau(i)\in\cdot|A_{1}^{i}=1)$, and $\Prob_{i}(\tau(i)\in\cdot|$ $A_{1}^{i}=-1)$, respectively, where $\phi_{\pm}:\equiv 1$ in the case $\Prob_{i}(A_{1}^{i}=\pm 1)=0$. Let $d\in\N$ denote the period (lattice span) of $\wh{\tau}(i)$.

\vspace{.1cm}
(a) If $\Prob_{i}(A_{1}^{i}=1)=1$, then $\wh{\tau}(i)=\tau(i)$ is aperiodic by our model assumptions.

\vspace{.1cm}
(b) If $\Prob_{i}(A_{1}^{i}=-1)=1$, then $\wh{\tau}(i)=\tau_{2}(i)$ implies $\phi(2\pi/d)^{2}=\wh{\phi}(2\pi/d)=1$, thus $\phi(2\pi/d)=-1$ or, equivalently, $\Prob_{i}(\tau(i)\in\frac{d}{2}+d\Z)=1$. Since $\tau(i)$ is integer-valued, $d$ must be even, and since $\tau(i)$ is aperiodic, $\frac{d}{2}$ and $d$ must be coprime, which together gives $d=2$.

\vspace{.1cm}
(c) Left with the case $0<\alpha:=\Prob_{i}(A_{1}^{i}=1)<1$, note that $p_{n}:=\Prob_{i}(\rho(1)=n)>0$ for all $n\in\N$ (see \eqref{eq:rho(1) aperiodic}). By combining this fact with
$$ 1\ =\ \wh{\phi}(2\pi/d)\ =\ p_{1}\phi_{+}(2\pi/d)+\sum_{n\ge 2}p_{n}\phi_{+}(2\pi/d)^{n-2}\phi_{-}(2\pi/d)^{2}, $$
we obtain $\phi(2\pi/d)=\alpha\phi_{+}(2\pi/d)+(1-\alpha)\phi_{-}(2\pi/d)=1$ and thus $d=1$, again by the aperiodicity of $\tau(i)$.\qed
\end{proof}

The next lemma completes the proof of Theorem \ref{thm:weak convergence} under the proviso $\Prob_{i}(|A_{1}^{i}|=1)=1$ for some/all $i\in\cS$ and further provides information about the possible limits $Q(i,\cdot)$. Recall from the beginning of this subsection that
$$ |\Pi_{n}|\ =\ \frac{a_{M_{0}}}{a_{M_{n}}}\quad\text{a.s.} $$
for a suitable positive sequence $(a_{j})_{j\in\cS}$, and from Lemma \ref{lem:sign Pi_n} that
\begin{equation}\label{eq:homology Pi_n}
\sign(\Pi_{n})\ =\ \frac{\sigma_{M_{n}}}{\sigma_{M_{0}}},\quad\text{hence}\quad\Pi_{n}\ =\ \frac{\sigma_{M_{n}}a_{M_{0}}}{\sigma_{M_{0}}a_{M_{n}}}\quad\text{a.s.}
\end{equation}
for a suitable sequence $(\sigma_{j})_{j\in\cS}$ in $\{\pm 1\}$ if even $\Prob_{i}(A_{1}^{i}=1)=1$ for some/all $i$. 

\vspace{.1cm}
If $\wh{\tau}(i)$ is 2-periodic and thus $\Prob_{i}(A_{1}^{i}=-1)=1$, then \eqref{eq:homology Pi_n} remains valid in a slightly modified form. This follows because Lemma \ref{lem:sign Pi_n} still applies, but to the augmented MRW $(\wh{M}_{n},\Pi_{n})_{n\ge 0}$ for which $\Prob_{i}(\Pi_{\wh{\tau}(i)}=1)=1$. We infer the existence of a suitable family $(\wh\sigma_{(i,\delta}))_{i\in\cS,\delta\in\{\pm 1\}}$ such that
$$ \sign(\Pi_{n})\ =\ \frac{\wh\sigma_{\wh{M}_{n}}}{\wh\sigma_{\wh{M}_{0}}},\quad\text{hence}\quad\Pi_{n}\ =\ \frac{\wh\sigma_{\wh{M}_{n}}a_{M_{0}}}{\wh\sigma_{\wh{M}_{0}}a_{M_{n}}}\quad\text{a.s.} $$
Moreover, $\wh\sigma_{(i,1)}=-\wh\sigma_{(i,-1)}$ for all $i\in\cS$ because $-1=\sign(\Pi_{\tau(i)})=\wh\sigma_{(i,-1)}/\wh\sigma_{(i,1)}$ $\Prob_{i}$-a.s.

\begin{Lemma}\label{lem:final lemma case 2}
Suppose \eqref{standing assumption}, $\Prob_{i}(|A_{1}^{i}|=1)=1$ for some $i\in\cS$, and let $Z_{0}$ be an admissible variable. Then $\Prob_{i}(\Psi_{1:n}(Z_{0})\in\cdot)$ converges weakly to some $Q_{i}$ iff \eqref{pi-degenerate} and one of the following conditions hold:
\begin{description}[(b.2)]\itemsep2pt
\item[(a)] $\Prob_{i}(\wh{\tau}(i)\in\cdot)$ is aperiodic. In this case,
\begin{align*}
Q_{i}\ =\ 
\begin{cases}
\sum_{j\in\cS}\pi_{j}\,\Prob_{i}(c_{i}+\sigma_{i}a_{i}(Z_{0}-c_{j})/(\sigma_{j}a_{j})\in\cdot),&\text{if }\Prob_{i}(A_{1}^{i}=1)=1,\\
\hfill\sum_{j\in\cS}\pi_{j}\,\Prob_{i}(c_{i}+a_{i}Y_{j}/a_{j}\in\cdot),&\text{if }\Prob_{i}(A_{1}^{i}=1)<1,
\end{cases}
\end{align*}
where $\Prob_{i}(Y_{j}\in\cdot)=[\Prob_{i}(Z_{0}-c_{j}\in\cdot)+\Prob_{i}(-(Z_{0}-c_{j})\in\cdot)]/2$.
\item[(b)] $\Prob_{i}(\wh{\tau}(i)\in\cdot)$ is 2-periodic and the weak limit of $\Pi_{2n}(Z_{0}-c_{M_{2n}})$ under $\Prob_{i}$ exists and is symmetric. In this case,
$$ Q_{i}\ =\ \frac{1}{2}\sum_{(j,\delta)\in\cS_{i}}\pi_{j}\,\Prob_{i}\left(c_{i}+\frac{\wh\sigma_{(i,1)}\,a_{i}}{\wh\sigma_{(j,\delta)}\,a_{j}}(Z_{0}-c_{j})\in\cdot\right), $$
where $\cS_{i}\subset\cS\times\{\pm 1\}$ denotes the cyclic class of the chain $(\wh{M}_{n})_{n\ge 0}$ which contains the state $(i,1)$.
\end{description}
\end{Lemma}

\begin{proof}
By Lemma \ref{lem:Psi_1:n(0) to infty in (b)}, the degeneracy condition \eqref{pi-degenerate} is necessary for the weak convergence of $\Prob_{i}(\Psi_{1:n}(Z_{0})\in\cdot)$ and therefore assumed hereafter.

\vspace{.1cm}
(a) If $\wh{\tau}(i)$ is aperiodic under $\Prob_{i}$, then we have that, for all $n\ge 0$,
$$ \Psi_{1:n}(Z_{0})\ =\ c_{i}\ +\ \sign(\Pi_{n})\,\frac{a_{i}}{a_{M_{n}}}\,(Z_{0}-c_{M_{n}})\quad\Prob_{i}\text{-a.s.} $$
and in the case $\Prob_{i}(A_{1}^{i}=1)=1$ even (by \eqref{eq:homology Pi_n})
$$ \Psi_{1:n}(Z_{0})\ =\ c_{i}\ +\ \frac{\sigma_{i}a_{i}}{\sigma_{M_{n}}a_{M_{n}}}\,(Z_{0}-c_{M_{n}})\quad\Prob_{i}\text{-a.s.} $$
An application of the ergodic theorem for either $(\wh{M}_{n})_{n\ge 0}$ or just $(M_{n})_{n\ge 0}$ yields the asserted weak convergence of $\Prob_{i}(\Psi_{1:n}(Z_{0})\in\cdot)$ and also the form of its limit $Q_{i}$.

\vspace{.1cm}
(b) If $\wh{\tau}(i)$ is 2-periodic under $\Prob_{i}$, then $(\Pi_{2n}(Z_{0}-c_{M_{2n}}))_{n\ge 0}$ forms a regenerative sequence with aperiodic regeneration epochs $\wh{\tau}_{n}(i)/2$, $n\ge 1$, and therefore converges in distribution under $\Prob_{i}$. A similar conclusion holds for $(\Pi_{2n+1}(Z_{0}-c_{M_{2n+1}}))_{n\ge 0}$. On  the other hand, the a.s. identity $\Psi_{1:n}(Z_{0})=c_{M_{0}}+\Pi_{n}(Z_{0}-c_{M_{n}})$, valid for all $n\ge 0$, shows that $\Psi_{1:n}(Z_{0})$ converges in distribution under $\Prob_{i}$ iff $\Pi_{n}(Z_{0}-c_{M_{n}})$ does so, thus in the present situation iff the weak limits under $\Prob_{i}$ of the afore-mentioned regenerative sequences are equal. We will finally verify that the latter conclusion holds iff the weak limit of $\Pi_{2n}(Z_{0}-c_{M_{2n}})$ is symmetric under $\Prob_{i}$.

Let us write $T_{n}\weakeq T_{n}'$ as shorthand for $T_{n}$ and $T_{n}'$ to have the same distributional limit as $n\to\infty$ (under some probability measure to be stated). Note that $\Prob_{i}(A_{1}^{i}=-1)=1$ by Lemma \ref{lem:lattice-type whtau(i)} and that $\Prob_{i}(\tau(i)=2\N-1)=1$ follows from the 2-periodicity of $\wh{\tau}(i)=\tau_{2}(i)$. Now we infer that under $\Prob_{i}$
\begin{align*}
\Pi_{2n+1}(Z_{0}-c_{M_{2n+1}})\ &\weakeq\ \1_{\{\tau(i)<2n+1\}}\left(-\prod_{k=\tau(i)+1}^{2n+1}A_{k}\right)(Z_{0}-c_{M_{2n+1}})\\
&\weakeq\ \left(-\prod_{k=\tau(i)+1}^{2n+1}A_{k}\right)(Z_{0}-c_{M_{2n+1}})\\
&\weakeq\ -\Pi_{2n}(Z_{0}-c_{M_{2n}}).
\end{align*}
Therefore, $\Pi_{2n}(Z_{0}-c_{M_{2n}})$ and $\Pi_{2n+1}(Z_{0}-c_{M_{2n+1}})$ have indeed the same weak limit under $\Prob_{i}$ iff the weak limit of $\Pi_{2n}(Z_{0}-c_{M_{2n}})$ is symmetric under $\Prob_{i}$. Since
$$ c_{i}+\Pi_{2n}(Z_{0}-c_{M_{2n}})\ =\ c_{i}+\frac{\wh\sigma_{(i.1)}\,a_{i}}{\wh\sigma_{(j,\sign(\Pi_{2n}))}\,a_{j}}(Z_{0}-c_{M_{2n}})\quad\Prob_{i}\text{-a.s.} $$
we also obtain the asserted form of $Q_{i}$ by letting $n$ tend to $\infty$.\qed
\end{proof}

\subsection{The case $\limsup_{n\to\infty}|\Pi_{\tau_{n}(i)}|=\infty$ a.s.}\label{subsec:cond (c)}

We start by noting that under the proviso of this subsection, Theorem \ref{thm:G/M theorem} by Goldie and Maller asserts that failure of \eqref{pi-degenerate} and thus of \eqref{degenerate} entails
\begin{equation}\label{eq:divergence of Psi^s(Z_{0})}
|\Psi_{1:n}^{i}(Z_{0})|\xrightarrow{\Prob_{i}}\infty,\quad\text{i.e.}\quad\lim_{n\to\infty}\Prob_{i}(|\Psi_{1:n}^{i}(Z_{0})|\le a)=0\quad\text{for all }a>0
\end{equation}
for all $i\in\cS$ and admissible $Z_{0}$. It may be surprising that there seems to be no easy argument to convert this into $|\Psi_{1:n}(Z_{0})|\xrightarrow{\Prob_{\pi}}\infty$. The result is shown as Lemma \ref{lem:Psi_1:n(0) to infty in (c)} below but requires the following auxiliary lemma.

\begin{Lemma}\label{lem:B_1 nondegenerate given M and Pi}
Suppose that $\limsup_{n\to\infty}|\Pi_{\tau_{n}(i)}|=\infty$ a.s. and \eqref{pi-degenerate} fails to hold. Then there exists $m\in\N$ such that the conditional law of $\Psi_{1:m}(0)$ given $M_{0},M_{m}$ and $\Pi_{m}$ is nondegenerate.
\end{Lemma}

\begin{proof}
Suppose by contraposition that $\Psi_{1:n}(0)=f_{n}(M_{0},M_{n},\Pi_{n})$ a.s. for all $n\in\N$ and fix an arbitrary $i\in\cS$. For any $n$ with $\Prob_{i}(M_{n}=i)>0$ and thus $\Prob_{i}(M_{n}=M_{2n}=i)>0$, we then obtain
$$ \Psi_{1:n}(0)+\Pi_{n}\Psi_{n+1:2n}(0)\ =\ \Psi_{1:2n}(0)\ =\ f_{2n}(i,i,\Pi_{2n})\quad\text{a.s.} $$
on $\{M_{0}=M_{n}=M_{2n}=i\}$. Hence, by Lemma \ref{lem:Grincevicius} (with $A_{1}=\Pi_{n},B_{1}=\Psi_{1:n}(0)$ and $B_{2}=\Psi_{n+1:2n}(0)$), either $\Psi_{1:n}(0)=b_{n}(1-\Pi_{n})$ or $(\Pi_{n},\Psi_{1:n}(0))=(1,c)$ $\Prob_{i}(\cdot|M_{n}=i)$-a.s. 

\vspace{.1cm}
If the first alternative holds for all $n\in I:=\{m:\Prob_{i}(M_{m}=i)>0\}$, then note first that
\begin{align*}
b_{mn}(1-\Pi_{mn})\ &=\ \Psi_{1:mn}(0)\ =\ \Psi_{1:n}(0)\ +\ \sum_{k=1}^{m-1}\Pi_{kn}\,\Psi_{kn+1:(k+1)n}(0)\\
&=\ b_{n}(1-\Pi_{n})\ +\ \sum_{k=1}^{m-1}\Pi_{kn}\,b_{n}\left(1-\frac{\Pi_{(k+1)n}}{\Pi_{kn}}\right)\ =\ b_{n}(1-\Pi_{mn})\quad\text{a.s.}
\end{align*}
on $\{M_{0}=M_{n}=\ldots=M_{mn}=i\}$ entails $b_{mn}=b_{n}$ for all $m\ge 1$. But for $m\in I$, the same argument shows $b_{m}=b_{mn}$ for all $n\ge 1$ and thus $b_{n}\equiv b\in\R$ for some $b\in\R$ and all $n\in I$ which in turn finally yields $B_{1}^{i}=\Psi_{1:\tau(i)}(0)=b(1-A_{1}^{i})$ $\Prob_{i}$-a.s. which is impossible if \eqref{pi-degenerate} is ruled out.

\vspace{.1cm}
If the second alternative holds for all $n\in I$, then we arrive at the conclusion that $\Prob_{i}(A_{1}^{i}=1)=1$ which is ruled out by the proviso of this subsection.

\vspace{.1cm}
Finally, consider the mixed case when $\Psi_{1:n}(0)=b_{n}(1-\Pi_{n})$ $\Prob_{i}(\cdot|M_{n}=i)$-a.s. and $(\Pi_{m},\Psi_{1:m}(0))=(1,b_{m})$ $\Prob_{i}(\cdot|M_{m}=i)$-a.s. for some distinct $m,n\in I$. Then use $\Psi_{1:ln}(0)=b_{n}(1-\Pi_{ln})$ $\Prob_{i}(\cdot|M_{ln}=i)$-a.s. for all $l\ge 1$ as shown above to infer
$$ b_{n}(1-\Pi_{lmn})\ =\ \Psi_{1:lmn}(0)\ =\ \ =\ \Psi_{1:m}(0)\ +\ \sum_{k=1}^{ln-1}\Psi_{km+1:(k+1)m}(0)\ =\ lnb_{m}\quad\text{a.s.} $$
on $\{M_{0}=M_{m}=\ldots=M_{lmn}=i\}$ for \emph{all} $l\ge 1$ and thus $b_{n}=0$ for all $n\in I$. But this means that $\Psi_{1:n}^{i}(0)=\Psi_{1:\tau_{n}(i)}(0)=0$ $\Prob_{i}$-a.s. for all $n\ge 1$ which is impossible by \eqref{eq:divergence of Psi^s(Z_{0})}.\qed
\end{proof}

\begin{Lemma}\label{lem:Psi_1:n(0) to infty in (c)}
Suppose that $\limsup_{n\to\infty}|\Pi_{\tau_{n}(i)}|=\infty$ a.s. Then $|\Psi_{1:n}(Z_{0})|\xrightarrow{\Prob_{\pi}}\infty$ for any admissible $Z_{0}$ iff \eqref{pi-degenerate} fails to hold.
\end{Lemma}

\begin{proof}
We must only verify that failure of \eqref{pi-degenerate} entails $|\Psi_{1:n}(Z_{0})|\xrightarrow{\Prob_{\pi}}\infty$ for any admissible $Z_{0}$. We fix an arbitrary $Z_{0}$. By Lemma \ref{lem:B_1 nondegenerate given M and Pi}, there exists $m\in\N$ such that $\Psi_{1:m}(0)$ is not a.s. constant given $M_{0},M_{m}$ and $\Pi_{m}$. The ensuing argument, for which we assume $m=1$ without loss of generality, will show that $|\Psi_{1:mn}(Z_{0})|\xrightarrow{\Prob_{\pi}}\infty$ as $n\to\infty$ which in turn is easily seen to be equivalent to $|\Psi_{1:n}(Z_{0})|\xrightarrow{\Prob_{\pi}}\infty$.

\vspace{.1cm}
Let $(B_{n}')_{n\ge 1}$ be a sequence on a possibly enlarged probability space which, when conditioned upon $(M_{n-1},A_{n})_{n\ge 1}$, is independent of $(B_{n})_{n\ge 1}$ and identically distributed. Hence, $Y_{n}:=B_{n}-B_{n}'$ forms a conditional symmetrization of $B_{n}$ given $M_{n-1},M_{n},A_{n}$ with nondegenerate conditional law. We claim that
\eqref{pi-degenerate} fails to hold for $(M_{n-1},A_{n},Y_{n})_{n\ge 1}$ as well. Namely, if it did, thus
$$ Y_{1}\ =\ c_{M_{0}}'-c_{M_{1}}'A_{1}\quad\text{a.s.} $$
for suitable constants $c_{j}'\in\cS$, then
$$ 0\ =\ Y_{1}^{j}\ :=\ \sum_{k=1}^{\tau(j)}\Pi_{k-1}Y_{k}\ =\ c_{j}'(1-A_{1}^{j})\quad\Prob_{j}\text{-a.s.} $$
for all $j\in\cS$ would follow by the symmetry of $Y_{\tau(j)}$ (clear by the conditional symmetry of the $Y_{n}$) and thereupon $c_{j}'=0$ for all $j\in\cS$ because $\Prob_{j}(A_{1}^{j}=1)=1$ is ruled out by the proviso of this subsection. But this would yield the contradiction $Y_{1}=0$ a.s. 

\vspace{.1cm}
By another appeal to Theorem \ref{thm:G/M theorem},
\begin{equation}\label{eq:divergence of symmetrization}
\left|\sum_{k=1}^{\tau_{n}(i)}\Pi_{k-1}Y_{k}\right|\ =\ \big|\Psi_{1:n}^{i}(Z_{0})-\Psi_{1:n}^{\prime\,i}(Z_{0})\big|\ \xrightarrow{\Prob_{\pi}}\ \infty,
\end{equation}
where $\Psi_{n}'(x)=A_{n}x+B_{n}'$ for $n\ge 1$. The elementary renewal theorem provides us with $n^{-1}N(n)\to[\Erw_{i}\tau(i)]^{-1}=\pi_{i}$ a.s. and so
\begin{equation}\label{eq:conv N(n)}
\lim_{n\to\infty}\Prob_{\pi}(N(n)\ge bn)\ =\ 1
\end{equation}
for $b:=\pi_{i}/2$. For $n\ge 1$, we further define $\wh{\Psi}_{n}=\Psi_{n}'\1_{\{\tau_{\lceil bn\rceil}(i)\ge n\}}+\Psi_{n}\1_{\{\tau_{\lceil bn\rceil}(i)<n\}}$, giving
\begin{align*}
\wh{\Psi}_{1:n}(Z_{0})\ =\ \sum_{k=1}^{\tau_{\lceil bn\rceil}(i)}\Pi_{k-1}B_{k}'\,+\,\big(\Psi_{1:n}(0)-\Psi_{1:\lceil bn\rceil}^{i}(0)\big)\,+\,\Pi_{n}Z_{0}.
\end{align*}
Given $\cG_{n}:=\sigma(Z_{0},M_{0},(M_{k},A_{k})_{1\le k\le n},N(n),\Psi_{1:n}(0)-\Psi_{1:\lceil bn\rceil}^{i}(0))$, $\Psi_{1:n}(Z_{0})$ and $\wh{\Psi}_{1:n}(Z_{0})$ are obviously iid on $\{N(n)\ge\lceil bn\rceil\}$. Using this facts in combination with Jensen's inequality, we obtain for all $x\ge 0$
\begin{align*}
\Prob_{\pi}&\left(\left|\sum_{k=1}^{\tau_{\lceil bn\rceil}(i)}\Pi_{k-1}Y_{k}\right|\le 2x\right)\\
&\ge\ \Prob_{\pi}\left(\left|\sum_{k=1}^{\tau_{\lceil bn\rceil}(i)}\Pi_{k-1}B_{k}-\sum_{k=1}^{\tau_{\lceil bn\rceil}(i)}\Pi_{k-1}B_{k}'\right|\le 2x,\,N(n)\ge bn\right)\\
&=\ \Prob_{\pi}\left(\big|\Psi_{1:n}(Z_{0})-\wh{\Psi}_{1:n}(Z_{0})\big|\le 2x,\,N(n)\ge bn\right)\\
&\ge\ \Prob_{\pi}\left(\big|\Psi_{1:n}(Z_{0})-\wh{\Psi}_{1:n}(Z_{0})\big|\le 2x,\,|\wh{\Psi}_{1:n}(Z_{0})|\le x,\,N(n)\ge bn\right)\\
&\ge\ \Erw_{\pi}\left(\1_{\{N(n)\ge\lceil bn\rceil\}}\,\Prob\left(|\Psi_{1:n}(Z_{0})|\le x,\,|\wh{\Psi}_{1:n}(Z_{0})|\le x\big|\cG_{n}\right)\right)\\
&=\ \Erw_{\pi}\left(\1_{\{N(n)\ge\lceil bn\rceil\}}\,\Prob\left(|\Psi_{1:n}(Z_{0})|\le x\big|\cG_{n}\right)^{2}\right)\\
&\ge\ \left[\Prob_{\pi}\left(|\Psi_{1:n}(Z_{0})|\le x,\,N(n)\ge\lceil bn\rceil\right)\right]^{2}\\
&\ge\ \left[\Prob_{\pi}\left(|\Psi_{1:n}(Z_{0})|\le x\right)\right]^{2}\,-\,\Prob_{\pi}\left(N(n)<\lceil bn\rceil\right)^{2}
\end{align*}
and then by use of \eqref{eq:divergence of symmetrization} and \eqref{eq:conv N(n)}
$$ \lim_{n\to\infty}\Prob_{\pi}\left(|\Psi_{1:n}(Z_{0})|\le x\right)\ =\ 0 $$
which proves $|\Psi_{1:n}(Z_{0})|\xrightarrow{\Prob_{\pi}}\infty$.\qed
\end{proof}

The proof of Theorem \ref{thm:weak convergence}(c) is now completed by the next lemma.
Given $\limsup_{n\to\infty}|\Pi_{\tau_{n}(i)}|=\infty$ a.s., either $\Pi_{n}\xrightarrow{\Prob_{i}}0$ or $\limsup_{n\to\infty}\Prob_{i}(|\Pi_{n}|>a)>0$ for some $a>0$. Note also that the first alternative is equivalent to $\Pi_{n}\xrightarrow{\Prob_{\pi}}0$.

\begin{Lemma}\label{lem:(B3)}
Suppose that $\limsup_{n\to\infty}|\Pi_{\tau_{n}(i)}|=\infty$ a.s. for some $i\in\cS$ and let $Z_{0}$ be admissible. Then $\Prob_{i}(\Psi_{1:n}(Z_{0})\in\cdot)$ converges weakly to some $Q_{i}$ iff \eqref{pi-degenerate} and one of the following conditions is satisfied:
\begin{description}[(c.2)]\itemsep2pt
\item[(a)] $\Pi_{n}\xrightarrow{\Prob_{i}} 0$. In this case, $Q_{i}=\delta_{c_{i}}$ and $\Psi_{1:n}(Z_{0})\xrightarrow{\Prob_{\pi}}c_{M_{0}}$.
\item[(b)] $\limsup_{n\to\infty}\Prob_{i}(|\Pi_{n}|>a)>0$ for some $a>0$, $c_{j}=c$ for all $j\in\cS$ and $Z_{0}=c$ $\Prob_{i}$-a.s. for some $c\in\R$. In this case, $Q_{i}=\delta_{c}$ and $\Psi_{1:n}(Z_{0})=c$ $\Prob_{i}$-a.s. for all $n\ge 0$.
\end{description}
\end{Lemma}

\begin{proof}
By the previous result, \eqref{pi-degenerate} is necessary for the distributional convergence of $\Psi_{1:n}(Z_{0})$ under $\Prob_{i}$. Since
$$ \Psi_{1:n}(Z_{0})\ \weakeq\ \1_{\{\tau(i)<n\}}\,\Psi_{1}^{i}\big(\Psi_{\tau(i)+1:n}(Z_{0})\big)\quad\text{under }\Prob_{i} $$
and $\1_{\{\tau(i)=k\}}\,\Psi_{1}^{i}$ and $\Psi_{k+1:n}(Z_{0})$ are independent under $\Prob_{i}$ for all $k,n\in\N$ with $k<n$, it follows easily that a possible limit $Q_{i}$ must solve the SFPE $R\eqdist\Psi_{1}^{i}(R)$, where $R$ and $(A_{1}^{i},B_{1}^{i})$ are independent. Hence $Q_{i}=\delta_{c_{i}}$ by an appeal to Theorem \ref{thm:fixed points IID}(c).

\vspace{.1cm}
(a) If $\Pi_{n}\xrightarrow{\Prob_{\pi}}0$, then $\Psi_{1:n}(Z_{0})=c_{M_{0}}+\Pi_{n}(Z_{0}-c_{M_{n}})$ converges in probability to $c_{M_{0}}$ under $\Prob_{\pi}$ by Slutsky's theorem because $(c_{M_{n}})_{n\ge 0}$ is stationary under $\Prob_{\pi}$.

\vspace{.1cm}
(b) If $\limsup_{n\to\infty}\Prob_{i}(|\Pi_{n}|>a)>0$ for some $a>0$, pick $j\in\cS$ and $\eps\in (0,1)$ such that
$$ \limsup_{n\to\infty}\Prob_{i}(|\Pi_{n}|>\eps,\,M_{n}=j)\,>\,0. $$
By the proviso of this subsection (valid for any $j\in\cS$ by solidarity, see Proposition \ref{prop:trichotomy embedded RW}), for all $x>0$ there exists $m(x)\in\N$ such that
$$ \Prob_{j}(|\Pi_{m(x)}|>x/\eps,\,M_{m(x)}=j)\,>\,0, $$
hence $\limsup_{n\to\infty}\Prob_{i}(|\Pi_{n}|>x,M_{n}=j)>0$ for all $x>0$ which in turn is easily seen to imply the very same for \emph{all} $(j,x)\in\cS\times (0,\infty)$. We infer from $Q_{i}=\delta_{c_{i}}$ that
$$ \Psi_{1:n}(Z_{0})-c_{i}\ =\ \Pi_{n}(Z_{0}-c_{M_{n}})\ \xrightarrow{\Prob_{i}}\ 0 $$
must be satisfied. Now assuming $\Prob_{i}(Z_{0}=c_{j})<1$ for some $j\in\cS$, we arrive at a contradiction via
\begin{align*}
0\ &=\ \lim_{n\to\infty}\Prob_{i}(|\Pi_{n}(Z_{0}-c_{M_{n}})>\eps)\\
&\ge\ \limsup_{n\to\infty}\Prob_{i}(|\Pi_{n}(Z_{0}-c_{j})|>\eps,\,M_{n}=j)\\
&=\ \limsup_{n\to\infty}\Prob_{i}(|\Pi_{n}|>\eps/|Z_{0}-c_{j}|,\,Z_{0}=c_{j},\,M_{n}=j)\\
&=\ \limsup_{n\to\infty}\Prob_{i}(|\Pi_{n}|>\eps/|Z_{0}-c_{j}|,\,M_{n}=j)\,\Prob_{i}(Z_{0}=c_{j})\ >\ 0,
\end{align*}
where the admissibility of $Z_{0}$ has been utilized for the last equality. Consequently, $\Prob_{i}(Z_{0}=c_{j})=1$ for all $j\in\cS$, i.e. $c_{j}\equiv c$ for some $c\in\R$ and $Z_{0}=c$ $\Prob_{i}$-a.s.\qed
\end{proof}

\section{Proof of Theorem \ref{thm:fixed-point property}}\label{sec:proof 3rd thm}

Recall that $\Psi_{1}P(i,\cdot)$ for a kernel $P\in\cP(\cS,\R)$ and $i\in\cS$ is defined as the conditional law of $A_{1}R'+B_{1}$ given $M_{0}=i$, where $R'$ denotes a random variable such that
\begin{equation}\label{eq:specification of R'}
\cL(R'|M_{0},M_{1},A_{1},B_{1})\ =\ \cL(R'|M_{0},M_{1})\ =\ P(M_{1},\cdot)\quad\text{a.s.}
\end{equation}
If $\Psi_{1}P=P$, then $P$ is called a solution to this equation or a fixed point of $\Psi_{1}$. Equivalent to this property is that (see \eqref{SFPE2})
\begin{align*}
R^{i}\ \eqdist\ \Psi_{1}^{i}(R^{i})\ =\ A_{1}^{i}R^{i}+B_{1}^{i}\quad\text{under }\Prob_{i}
\end{align*}
for all $i\in\cS$, where, for each $j\in\cS$ and under each $\Prob_{i}$, $R^{j}$ has distribution $P(j,\cdot)$ and is independent of all other occurring random variables. The solutions to these ordinary SFPE are characterized in Theorem \ref{thm:fixed points IID}, and this result will therefore repeatedly be used herafter.

\subsection{The case (C1)}\label{subsec:(C1)}

Assuming $\Pi_{\tau_{n}(i)}\to 0$ a.s., let $P$ denote a fixed point of $\Psi_{1}$, so that $P(i,\cdot)$ solves \eqref{SFPE2} for any $i\in\cS$. By Theorem \ref{thm:fixed points IID}(a), we infer $\Erw_{i}J_{i}(\log^{+}|B_{1}^{i}|)<\infty$ as well as $P(i,\cdot)=\Prob_{i}(\lim_{n\to\infty}\Psi_{1:n}^{i}(0)\in\cdot)$ which in turn equals the law of $Z_{\infty}$ by Theorem \ref{thm:a.s. convergence} for any $i$.

\vspace{.1cm}
Conversely, if $\Erw_{i}J_{i}(\log^{+}|B_{1}^{i}|)<\infty$, then $P\in\cP(\cS,\R)$, defined by $P(i,\cdot):=\Prob_{i}(Z_{\infty}\in\cdot)$, is a fixed point of $\Psi_{1}$, and it is unique because, by Theorem \ref{thm:weak convergence}(a), $\Psi_{1:n}(Z_{0})\idist Z_{\infty}$ under any $\Prob_{i}$ and for any admissible $Z_{0}$.

\vspace{.2cm}
For the remaining cases where $\Pi_{\tau_{n}(i)}$ does not converge to 0 a.s., Theorem \ref{thm:fixed points IID} implies that the degeneracy condition \eqref{degenerate} and therefore \eqref{pi-degenerate} must be satisfied for the existence of a fixed point $P\in\cP(\cS,\R)$ of $\Psi_{1}$ and thus of \eqref{SFPE2} for \emph{all} $i\in\cS$.

\subsection{The case (C2)}\label{subsec:(C2)}

Assume $\Prob_{i}(A_{1}^{i}=1)<\Prob_{i}(|A_{1}^{i}|=1)=1$ for all $i\in\cS$ and \eqref{pi-degenerate}. Recall that from Subsection \ref{subsec:cond (b)} that $|\Pi_{n}|=a_{M_{0}}/a_{M_{n}}$ a.s. for all $n\ge 0$ and a suitable sequence $(a_{i})_{i\in\cS}$ of positive numbers. Given a fixed point $P$ of $\Psi_{1}$, it follows from Theorem \ref{thm:fixed points IID} that $P(i,\cdot)$ is symmetric about $c_{i}$ for any $i\in\cS$. In other words,
$$ X^{i}\ :=\ R^{i}-c_{i} $$
is a symmetric random variable. By Proposition \ref{prop:degeneracy consequences} (see \eqref{eq:general homology}),
\begin{equation}\label{eq:SFPE X^i}
X^{i}\ =\ \Psi_{1:n}(R^{M_{n}})-c_{i}\ =\ \Psi_{1:n}(R^{M_{n}})-\Psi_{1:n}(c_{M_{n}})\ =\ \Pi_{n}X^{M_{n}}\quad\Prob_{i}\text{-a.s.}
\end{equation}
for all $n\ge 0$ and $i\in\cS$. In particular,
$$ |X^{i}|\ \eqdist\ |\Pi_{n}|\,|X^{M_{n}}|\ =\ a_{i}\,\frac{|X^{M_{n}}|}{a_{M_{n}}}\ =:\ X_{n}'\quad\text{under }\Prob_{i}. $$
The ergodicity of $(M_{n})_{n\ge 0}$ implies the $\Prob_{i}$-distributional convergence of $X_{n}'$ to some $X'$ which does not depend on $i$. As a consequence,
$$ R^{i}\ \eqdist\ a_{i}X\,+\,c_{i}\quad\text{under }\Prob_{i} $$
for any $i\in\cS$, where $\Prob_{i}(X\in\cdot)$ is symmetric and the same for each $i$.

\vspace{.1cm}
Conversely, we must show that $P\in\cP(\cS,\R)$ is a fixed point of $\Psi_{1}$ whenever
$$ P(i,\cdot)\ =\ \Prob(a_{i}X+c_{i}\in\cdot),\quad i\in\cS, $$
for a symmetric random variable $X$, $(c_{i})_{i\in\cS}$ given by \eqref{pi-degenerate} and $(a_{i})_{i\in\cS}$ as before. Let $R'$ denote a random variable satisfying \eqref{eq:specification of R'}. Under $\Prob_{i}$ for any $i\in\cS$, we then obtain by use of \eqref{pi-degenerate}
\begin{align*}
\Psi_{1}(R')\ &=\ A_{1}R'+c_{i}-A_{1}c_{M_{1}}\ =\ c_{i}+A_{1}(R'-c_{M_{1}})\ \eqdist\ c_{i}+|A_{1}|(R'-c_{M_{1}})\\
&=\ c_{i}+\frac{a_{i}}{a_{M_{1}}}(R'-c_{M_{1}})\ \eqdist\ c_{i}+a_{i}X,
\end{align*}
which is the desired result.

\subsection{The case (C3)}\label{subsec:(C3)}

Now assume $\Prob_{i}(A_{1}^{i}=1)=1$ for all $i\in\cS$ and \eqref{pi-degenerate}, thus $\Prob_{i}(B_{1}^{i}=0)=1$ as well. If $P$ denotes a fixed point of $\Psi_{1}$, then Theorem \ref{thm:fixed points IID} asserts that $P(i,\cdot)$ can be an arbitrary distribution for each $i\in\cS$. By Proposition \ref{prop2:degeneracy consequences}(c), there exists an infinite class of sequences $(c_{i})_{i\in\cS}$ for which \eqref{pi-degenerate} holds, thus $\Psi_{1}(c_{M_{1}})=c_{M_{0}})$ a.s. For any such sequence,
$$ X^{i}\ :=\ R^{i}-c_{i}\ \eqdist\ \Pi_{n}X^{M_{n}} $$
remains true under $\Prob_{i}$ (cf. \eqref{eq:SFPE X^i}). By Lemma \ref{lem:sign Pi_n},
$$ X^{i}\ \eqdist\ \frac{a_{i}\sigma_{i}}{a_{M_{n}}\sigma_{M_{n}}}\,X^{M_{n}}\quad\text{under $\Prob_{i}$ for all }n\ge 0. $$
By another appeal to the ergodic theorem, $X^{M_{n}}/(a_{M_{n}}\sigma_{M_{n}})$ converges in distribution under each $\Prob_{i}$ to some random variable $X$ whose law does not depend on $i$. Hence
$$ R^{i}\ \eqdist\ a_{i}\sigma_{i}X\,+\,c_{i} $$
under $\Prob_{i}$ for all $i\in\cS$.

\vspace{.1cm}
Conversely, if $P(i,\cdot)=\Prob_{i}(a_{j}\sigma_{j}X+c_{j}\in\cdot)$ for a random variable $X$ with arbitrary law which does not depend on $i$ and if $R'$ satisfies \eqref{eq:specification of R'}, then $\Psi_{1}P=P$ follows from
$$ \Psi_{1}(R')\ =\ A_{1}(R'-c_{M_{1}})+c_{i}\ =\ \frac{a_{i}\sigma_{i}}{a_{M_{1}}\sigma_{M_{1}}}(R'-c_{M_{1}})+c_{i}\ \eqdist\ a_{i}\sigma_{i}X+c_{i} $$
under $\Prob_{i}$ for all $i\in\cS$.

\subsection{The case (C4)}\label{subsec:(C4)}

If $\limsup_{n\to\infty}|\Pi_{\tau_{n}(i)}|=\infty$ a.s. and \eqref{pi-degenerate} holds, then $(c_{i})_{i\in\cS}$ is uniquely determined by Proposition \ref{prop:degeneracy consequences}. Moreover, if $\Psi_{1}P=P$, then Theorem \ref{thm:fixed points IID} provides us with $P(i,\cdot)=\delta_{c_{i}}$ for all $i\in\cS$. Conversely, any $P\in\cP(\cS,\R)$ having the latter property is also a fixed point of $\Psi_{1}$ by \eqref{eq:general homology} of Proposition \ref{prop:degeneracy consequences}.\qed

\section{Proof of Theorem \ref{thm:convergence forward}}\label{sec:proof thm forward}

In the following, let $i\in\cS$ be fixed and $Z_{0}$ be an admissible variable for $(M_{n},A_{n+1},B_{n+1})_{n\ge 0}$.

\vspace{.2cm}
(a) Suppose that $\Pi_{\tau_{n}(i)}\to 0$ a.s. and $\Erw_{i}J_{i}(\log^{+}|{}^{\#}B_{1}|)<\infty$ and recall that $\Pi_{n}\xrightarrow{\Prob_{\pi}}0$ (see \eqref{eq:Pi_n to 0}) and thereupon
\begin{equation*}
|\Psi_{n:1}(Z_{0})-\Psi_{n:1}(0)|\ =\ |\Pi_{n}Z_{0}|\ \xrightarrow{\Prob_{\pi}}\ 0.
\end{equation*}
Hence, Theorem \ref{thm:a.s. convergence} in combination with \eqref{eq:forward<->backward dual} provides us with
$$ \Psi_{n:1}(Z_{0})\ \weakeq\ \Psi_{n:1}(0)\ \weakeq\ {}^{\#}\Psi_{1:n}(0)\ \to\ {}^{\#}Z_{\infty}\quad\Prob_{\pi}\text{-a.s.}, $$
and so $\Prob_{\pi}(\Psi_{n:1}(Z_{0})\in\cdot)\weakly\Prob_{\pi}({}^{\#}Z_{\infty}\in\cdot)$. In order to see that $\Prob_{i}(\Psi_{n:1}(Z_{0})\in\cdot)$ converges to the same limit, we make use of a coupling argument.

\vspace{.1cm}
Let $(M_{n}',\Psi_{n+1}')_{n\ge 0}$ be independent of $(M_{n},\Psi_{n+1})_{n\ge 0}$ and $Z_{0}$ under $\Prob_{\pi}$ with
$$ \Prob_{\pi}\big((M_{n}',\Psi_{n+1}')_{n\ge 0}\in\cdot\big)\ =\ \Prob_{i}\big((M_{n},\Psi_{n+1})_{n\ge 0}\in\cdot\big). $$
Then the coupling time $T:=\inf\{n:M_{n}=M_{n}'\}$ is a.s. finite and the coupling process
\begin{align*}
\big(\widetilde{M}_{n},\widetilde{\Psi}_{n+1}\big)\ :=\ 
\begin{cases}
(M_{n}',\Psi_{n+1}'),&\text{if }n<T,\\
\hfill (M_{n},\Psi_{n+1}),&\text{if }n\ge T
\end{cases}
\end{align*}
has the same law as $(M_{n}',\Psi_{n+1}')_{n\ge 0}$ while its post-$T$ sequence coincides with $(M_{n},\Psi_{n+1})_{n\ge T}$. This in combination with $\Pi_{n}\xrightarrow{\Prob_{\pi}}0$ (see \eqref{eq:Pi_n to 0}) implies
\begin{align*}
|\Psi_{n:1}(Z_{0})-\widetilde{\Psi}_{n:1}(Z_{0})|\ &=\ |\Psi_{n:T+1}(\Psi_{T:1}(Z_{0}))-\Psi_{n:T+1}(\Psi_{T:1}'(Z_{0}))|\\
&=\ \frac{\Pi_{n}}{\Pi_{T}}|\,\Psi_{T:1}(Z_{0})-\Psi_{T:1}'(Z_{0})|\ \xrightarrow{\Prob_{\pi}}\ 0
\end{align*}
and thus the desired result.

In order to proceed to the remaining cases, we first need the following lemma which forms the counterpart of Lemma \ref{lem:Psi_1:n(0) to infty in (b)}.

\begin{Lemma}\label{lem:Case (b) forward iteration}
Suppose that $\Prob_{i}(|A_{1}^{i}|=1)=1$. Then $|\Psi_{n:1}(Z_{0})|\xrightarrow{\Prob_{\pi}}\infty$ for any admissible $Z_{0}$ if \eqref{dual pi-degenerate} fails to hold.
\end{Lemma}

\begin{proof}
Since $|\Psi_{n:1}(Z_{0})-\Psi_{n:1}(0)|=|\Pi_{n}Z_{0}|$ for all $n\ge 1$ and $(\Pi_{n}Z_{0})_{n\ge 0}$ is tight under the given assumption (see at the beginning of Subsection \ref{subsec:cond (b)}), it suffices to consider the case when $Z_{0}=0$. Now, if \eqref{dual pi-degenerate} fails, then $|\Psi_{n:1}(0)|\eqdist |{}^{\#}\Psi_{1:n}(0)|\xrightarrow{\Prob_{\pi}}\infty$ by Lemma \ref{lem:Psi_1:n(0) to infty in (b)}.\qed
\end{proof}

So we see that, if $\Pi_{\tau_{n}(i)}$ does not converge to 0 a.s., then \eqref{dual pi-degenerate} is necessary for the weak convergence of $\Prob_{i}(\Psi_{n:1}(Z_{0})\in\cdot)$.
Recall that this condition implies \eqref{dual pi-degenerate iterate}, that is 
\begin{align*}
\Psi_{n:1}(Z_{0})\ =\ c_{M_{n}}+\Psi_{n:1}(Z_{0})-\Psi_{n:1}(c_{M_{0}})\ =\ c_{M_{n}}+\Pi_{n}(Z_{0}-c_{M_{0}})\quad\text{a.s.}
\end{align*}
for all $n\ge 1$.

\vspace{.2cm}
(b.1) If $\Prob_{i}(|A_{1}^{i}|=1)=1$, \eqref{dual pi-degenerate} holds and $\wh{\tau}(i)$ is aperiodic, then $\Pi_{n}=\sign(\Pi_{n})\,a_{M_{0}}/a_{M_{n}}$ a.s. and therefore
$$ \Psi_{n:1}(Z_{0})\ =\ \sign(\Pi_{n})\,\frac{a_{i}}{a_{M_{n}}}(Z_{0}-c_{i})+c_{M_{n}}\quad\Prob_{i}\text{-a.s.} $$
Arguing in a similar manner as in the proof of Lemma \ref{lem:final lemma case 2}, $\Prob_{i}(\Psi_{n:1}(Z_{0})\in\cdot)\weakly Q_{i}$ follows with
\begin{align}\label{eq1:def Q_i under (b.1)}
Q_{i}\ =\ \sum_{j\in\cS}\pi_{j}\,\Prob_{i}\left(c_{j}+\frac{\sigma_{i}a_{i}}{\sigma_{j}a_{j}}(Z_{0}-c_{i})\in\cdot\right)
\end{align}
if $\Prob_{i}(A_{1}^{i}=1)=1$, and
\begin{align}\label{eq2:def Q_i under (b.1)}
Q_{i}\ =\ \sum_{j\in\cS}\pi_{j}\,\Prob_{i}\left(c_{j}+\frac{a_{i}}{a_{j}}Y\in\cdot\right)
\end{align}
if $\Prob_{i}(A_{1}^{i}=1)<1$ and $\Prob_{i}(Y\in\cdot)=\frac{1}{2}[\Prob_{i}(Z_{0}-c_{i}\in\cdot)+\Prob_{i}(-(Z_{0}-c_{i})\in\cdot)]$.

\vspace{.1cm}
(b.2) If $\Prob_{i}(|A_{1}^{i}|=1)=1$, \eqref{dual pi-degenerate} holds and $\wh{\tau}(i)$ is 2-periodic, thus $\Prob_{i}(A_{1}^{i}=-1)=1$, the assertion follows again by a straightforward adaptation of the respective arguments in the proof of Lemma \ref{lem:final lemma case 2}. First one shows that $\Prob_{i}(\Psi_{n:1}(Z_{0})\in\cdot)$ converges weakly iff $\Prob_{i}(\Pi_{n}(Z_{0}-c_{i})\in\cdot)$ does so which in turn holds iff the weak limits of $c_{M_{2n}}\pm\Pi_{2n}(Z_{0}-c_{i})$ under $\Prob_{i}$ coincide. Then use
$$ c_{M_{2n}}+\Pi_{2n}(Z_{0}-c_{i})\ =\ c_{M_{2n}}+\frac{\wh\sigma_{(i.1)}\,a_{i}}{\wh\sigma_{(j,\sign(\Pi_{2n}))}\,a_{j}}(Z_{0}-c_{i})\quad\Prob_{i}\text{-a.s.} $$
to finally identify the weak limit of $\Prob_{i}(\Psi_{n:1}(Z_{0})\in\cdot)$ as
\begin{equation}\label{eq:def Q_i under (b.2)}
Q_{i}\ =\ \frac{1}{2}\sum_{j\in\cS_{i}}\pi_{j}\,\Prob_{i}\left(c_{j}+\frac{\wh{\sigma}_{(i,1)}a_{i}}{\wh{\sigma}_{(j,\delta)}a_{j}}(Z_{0}-c_{i})\in\cdot\right).
\end{equation}

(c.1) If $\limsup_{n\to\infty}|\Pi_{\tau_{n}(i)}|=\infty$ a.s. and $\Pi_{n}\xrightarrow{\Prob_{i}}0$, then \eqref{dual pi-degenerate iterate} obviously implies the assertion.

Before proceeding to the remaining cases, note further that $(\tau_{n}(i))_{n\ge 0}\eqdist ({}^{\#}\tau_{n}(i))_{n\ge 0}$ as well as $\wh{\tau}(i)\eqdist{}^{\#}\wh{\tau}(i)$ under $\Prob_{i}$,
where ${}^{\#}\tau_{n}(i)$ and ${}^{\#}\tau(i)$ have the obvious meaning.

\vspace{.1cm}
(c.2) Finally suppose $\limsup_{n\to\infty}|\Pi_{\tau_{n}(i)}|=\infty$ a.s. and $\limsup_{n\to\infty}\Prob_{i}(|\Pi_{n}|>a)>0$ for some $a>0$. Observe that, if $\Prob_{i}(\Psi_{n:1}(Z_{0})\in\cdot)$ converges weakly to some $Q_{i}$, then \eqref{dual pi-degenerate iterate} in combination with the tightness of $(c_{M_{n}})_{n\ge 0}$ entails the tightness of $(\Pi_{n}(Z_{0}-c_{0}))_{n\ge 0}$ under $\Prob_{i}$. But the latter is only possible if either $\Prob_{i}(Z_{0}=c_{i})=1$, in which case $Q_{i}=\Prob_{\pi}(c_{M_{0}}\in\cdot)$ (the ergodic limit of $c_{M_{n}}$) as claimed by the theorem, or $(|\Pi_{n}|)_{n\ge 0}$ is tight under $\Prob_{i}$. The subsequent argument will rule out the last alternative. 

\vspace{.1cm}
Recall that $S_{n}=-\log|\Pi_{n}|$, $n\ge 0$, is not null-homologous because $\limsup_{n\to\infty}|\Pi_{\tau_{n}(i)}|=\infty$ a.s. Conditioned upon the driving chain $(M_{n})_{n\ge 0}$, they form a partial sum sequence of independent random variables under $\Prob_{i}$. Since $|\Pi_{n}|=e^{-S_{n}}$, the tightness of $(|\Pi_{n}|)_{n\ge 0}$ is equivalent to the tightness of $(S_{n}^{-})_{n\ge 0}$ which in turn ensures that, for any $\eps>0$, there exists $x>\log a$ such that
$$ \sup_{n\ge 0}\,\Prob_{i}(S_{n}\le -x)\ <\ \eps. $$
Lemma \ref{lem:concentration MRW} in the Appendix further provides us with $|S_{n}|\xrightarrow{\Prob_{\pi}}\infty$, in particular
$$ \lim_{n\to\infty}\Prob_{i}(|S_{n}|\le x)\ =\ 0, $$
whence, for suitable $m=m(\eps,x)\in\N$,
$$ \sup_{n\ge m}\,\Prob_{i}(|\Pi_{n}|>a)\ \le\ \sup_{n\ge m}\,\Prob_{i}(|\Pi_{n}|\ge e^{-x})\ =\ \sup_{n\ge m}\,\Prob_{i}(S_{n}\le x)\ \le\ \eps. $$
As $\eps>0$ was arbitrarily chosen, we arrive at $\limsup_{n\to\infty}\Prob_{i}(|\Pi_{n}|>a)=0$ and thus a contradiction.\qed

\section{Appendix}\label{sec:appendix}

\begin{Lemma}\label{lem:regenerative process}
Let $(M_{n},X_{n})_{n\ge 0}$ be a Markov-modulated sequence with ergodic driving chain $(M_{n})_{n\ge 0}$ on a countable state space $\cS$ and $X_{0},X_{1},\ldots$ taking values in $\R^{d}$ for some $d\ge 1$. For an arbitrary $i\in\cS$ and a measurable, real-valued function $f$, define the sequence
$$ T_{n}\ :=\ f(X_{\wh{\tau}_{n-1}(i)+1},\ldots,X_{\wh{\tau}_{n}(i)}), $$
where $(\wh{\tau}_{n}(i))_{n\ge 0}$ is a subsequence of $(\tau_{n}(i))_{n\ge 0}$ having iid integrable increments under $\Prob_{i}$. Then $T_{\wh{N}(n)+1}$ converges in distribution (under any $\Prob_{j}$) to the size-biased distribution
$$ \frac{1}{\Erw_{i}\wh{\tau}(i)}\Erw_{i}\wh{\tau}(i)\1_{\{T_{1}\in\cdot\}}, $$
where $\wh{\tau}(i):=\wh{\tau}_{1}(i)$ and $\wh{N}(n):=\sup\{k\ge 1:\wh{\tau}_{k}(i)\le n\}$.
\end{Lemma}

\begin{proof}
This can be proved by standard renewal arguments and we therefore refrain from giving further details.\qed
\end{proof}

Our last lemma has been crucial for the proof of Theorem \ref{thm:convergence forward} in the case (c.2) but is of interest also in its own right.

\begin{Lemma}\label{lem:concentration MRW}
Let $(M_{n},S_{n})_{n\ge 0}$ be a MRW which is not null-homologous and has positive recurrent  driving chain $M=(M_{n})_{n\ge 0}$ with stationary distribution $\pi$. Then $|S_{n}|\xrightarrow{\Prob_{\pi}}\infty$, i.e.
\begin{equation}\label{eq:concentration MRW}
\lim_{n\to\infty}\Prob_{\pi}(|S_{n}|\le x)\ =\ 0
\end{equation}
for all $x>0$.
\end{Lemma}

\begin{proof}
Let $X_{1},X_{2},\ldots$ denote the increments of $(S_{n})_{n\ge 0}$. We consider two cases:

\vspace{.1cm}
\textsc{Case 1}. $S_{n}=\theta_{n}(M_{0},M_{n})$ a.s. for all $n\in\N$ and suitable $\theta_{n}:\cS^{2}\to\R$. Then
$$ \theta_{n}(M_{0},M_{n})\ =\ \sum_{k=1}^{n}\theta_{1}(M_{k-1},M_{k})\quad\text{a.s.} $$
for all $n\in\N$, and this further implies $S_{\tau(i)}=\theta_{\tau(i)}(i,i)$ $\Prob_{i}$-a.s. as well as
\begin{equation}\label{eq:theta_m(i,i)}
\theta_{k}(i,i)+\theta_{l}(i,i)\ =\ \theta_{k+l}(i,i)
\end{equation}
for all $i\in\cS$ and $k,l\in\cC_{i}:=\{m\in\N:\Prob_{i}(M_{m}=i)>0\}$. Fix any $i\in\cS$. Since $(M_{n},S_{n})_{n\ge 0}$ is not null-homologous, we have $\Prob_{i}(S_{\tau(i)}\ne 0)>0$, w.l.o.g. $\Prob_{i}(S_{\tau(i)}>0)>0$, which in turn entails $\theta_{m}(i,i)>0$ for some $m\in\cC_{i}$. Now use \eqref{eq:theta_m(i,i)} and $m\cC_{i}\subset\cC_{i}$ to infer
$$ \theta_{km}(i,i)\ =\ m\,\theta_{k}(i,i)\ =\ k\,\theta_{m}(i,i)\ >\ 0 $$
and thereby $\theta_{k}(i,i)>0$ for all $k\in\N$. As a consequence, $S_{\tau(i)}>0$ $\Prob_{i}$-a.s. and so $S_{\tau_{n}(i)}\to\infty$ a.s. Defining $D_{n}^{i}:=\max_{\tau_{n-1}(i)\le k\le\tau_{n}(i)}|S_{k}|$ for $n\ge 1$, clearly iid under $\Prob_{i}$, Lemma \ref{lem:regenerative process} ensures that $D_{N(n)}^{i}$ converges in distribution under any $\Prob_{j}$, where as earlier $N(n)=\sup\{k:\tau_{n}(i)\le n\}$. By combining this with $|S_{n}|\le S_{\tau_{N(n)}(i)}+D_{N(n)}^{i}$ a.s. and $S_{\tau_{N(n)}(i)}\to\infty$ a.s., we finally arrive at the conclusion that $S_{n}\xrightarrow{\Prob_{\pi}}\infty$. 

\vspace{.2cm}
\textsc{Case 2}. For some $m\in\N$, $S_{m}$ given $M_{0},M_{m}$ is not a.s. constant. Then put $X_{n}^{(m)}:=S_{nm}-S_{(n-1)m}$ for $n\in\N$ and $M^{(m)}:=(M_{nm})_{n\ge 0}$. Note that, if
$$ H(X_{n}^{(m)},\lambda|M^{(m)})\ :=\ \sup_{x\in\R}\Prob_{\pi}(x\le X_{n}^{(m)}\le x+\lambda|M^{(m)}),\quad\lambda>0, $$
denotes the conditional concentration function of $X_{n}^{(m)}$ given $M^{(m)}$ under $\Prob_{\pi}$, then
$$ H(X_{n}^{(m)},\lambda|M^{(m)})\ =\ H(X_{n}^{(m)},\lambda|M_{(n-1)m},M_{nm})\quad\text{a.s.} $$
Since the $X_{n}^{(m)}$ are conditionally independent given $M^{(m)}$, the Kolmogorov-Rogozin inequality (see e.g. Ess\'een \cite{Esseen:68}) provides us with
\begin{align}\label{eq:K-R-inequality}
H(S_{n},\lambda|M^{(m)})\ \le\ \frac{C}{l^{1/2}}\left(\frac{1}{l}\sum_{k=1}^{l}H(X_{k}^{(m)},\lambda|M_{(k-1)m},M_{km})\right)^{-1/2}\quad\text{a.s.}
\end{align}
if $n=lm+r$ for $l\in\N$ and $r\in\{0,\ldots,m-1\}$. By the ergodic theorem,
\begin{align*}
\lim_{l\to\infty}\,\frac{1}{l}\sum_{k=1}^{l}H(X_{k}^{(m)},\lambda|M_{(k-1)m},M_{km})\ =\ \Erw_{\pi}H(S_{m},\lambda|M_{0},M_{m})\quad\text{a.s.}
\end{align*}
for all $\lambda>0$. Consequently, $\Prob_{\pi}(|S_{n}|\le x|M^{(m)})\to 0$ a.s. and particularly \eqref{eq:concentration MRW} for all $x>0$ follows from \eqref{eq:K-R-inequality} because $\Erw_{\pi}H(S_{m},\lambda|M_{0},M_{m})>0$ for some $\lambda>0$ under the given assumption that $S_{m}$ given $M_{0},M_{m}$ is not a.s. constant.\qed
\end{proof}

\bibliographystyle{abbrv}
\bibliography{StoPro}

\end{document}